\RequirePackage{fix-cm}
\documentclass[smallextended]{svjour3}       
\smartqed  
 \usepackage{mathptmx}      
\usepackage{amsmath}
 \usepackage{mathrsfs}
 \usepackage{stmaryrd}
 \usepackage{amssymb}
\usepackage{bm}
\usepackage{subfig}
\usepackage{graphicx}
\usepackage{float}
\newcommand{\norm}[1]{\lVert #1 \rVert}

\newcommand{\tnorm}[1]{\ensuremath{\left| \! \left| \! \left|} #1 \ensuremath{\right| \! \right| \! \right|}}
\newcounter{dummy} \numberwithin{dummy}{section}
\newtheorem{thm}[dummy]{Theorem}

\newtheorem{lm}[dummy]{Lemma}
\newtheorem{remk}[dummy]{Remark}
\numberwithin{equation}{section} 
\numberwithin{example}{section} 

\begin{document}
\title{Error analysis of nonconforming and  mixed FEMs for second-order linear non-selfadjoint and indefinite elliptic problems
}

\titlerunning{Nonconforming and Mixed FEMs}        

\author{Carsten Carstensen     \and      Asha K. Dond     \and \\ Neela Nataraj \and Amiya K. Pani }

\authorrunning{Carstensen} 

\institute{Carsten Carstensen \at
              Department of Mathematics, Humboldt-Universit\"{a}t zu Berlin, 10099 Berlin, Germany \\
              \email{cc@math.hu-berlin.de}           
           \and
          Asha K. Dond, ~~ Neela Nataraj, ~~Amiya K. Pani\at
              Department of Mathematics, Industrial Mathematics Group, IIT Bombay, Powai, Mumbai-400076\\
              \email{asha@math.iitb.ac.in},  {neela@math.iitb.ac.in},  {akp@math.iitb.ac.in}        
}

\date{Received: date / Accepted: date}

\maketitle

\begin{abstract}
The state-of-the art proof of a global inf-sup condition on mixed finite element schemes
does not allow for an analysis of  truly indefinite,  second-order linear  elliptic PDEs. This 
 paper, therefore, first analyses a nonconforming finite element discretization 
which converges owing to some
{\it a~ priori}  $L^2$ error estimates even for reduced regularity on non-convex polygonal domains.
An equivalence result of that nonconforming finite element scheme to the mixed finite element method (MFEM) leads to the well-posedness 
  of the discrete solution
and to {\it a~priori} error estimates for the MFEM. The explicit residual-based {\it a~posteriori} error analysis allows some reliable and efficient error control
 and motivates some adaptive discretization which  improves the empirical convergence rates in three computational benchmarks.
\keywords{  non-selfadjoint, indefinite linear elliptic problems \and stability \and nonconforming FEM \and mixed FEM  \and 
equivalence of RTFEM and NCFEM \and a priori error estimates\and residual-based {\it a~ posteriori} error analysis }
\end{abstract}

\section{Introduction}
\label{intro}
The general second-order linear elliptic PDE  
on a  simply-connected bounded  polygonal Lipschitz domain  $\Omega \subset {\mathbb R}^2$ with boundary $\partial \Omega$ reads for given right-hand side $f\in L^2(\Omega)$ as 
\begin{eqnarray}\label{eq1}
 \mathcal{L}u :=  -\nabla \cdot (\mathbf A\nabla u+u {\mathbf b})+ \gamma \: u=f
 \hspace{.5cm}\mbox {in}~\Omega \;\;\;\;\;u=0\hspace{.5cm}\mbox{on}~\partial\Omega.
\end{eqnarray}
The coefficients are all essentially bounded functions and the eigenvalues of the
symmetric matrix ${\mathbf A}$ are all positive and uniformly  bounded away from zero. 
 The point is that the convective term ${\bf {b}}$ and the reaction
term $\gamma$ may be arbitrary as long as the boundary value problem \eqref{eq1} is well-posed in the sense
that zero is not an eigenvalue. In other words, $\mathcal{L}: H^{1}_0(\Omega)\to  H^{-1}(\Omega)$ is  supposed to be injective, 
where $H^{-1}(\Omega)$ is the dual space of $H^1_0(\Omega):=\{v\in H^1(\Omega): v|_{\partial \Omega}=0\}.$  Since $\mathcal{L}$ is a bounded linear operator
 between Hilbert spaces, this is equivalent to assume that $\mathcal{L}$ is an isomorphism.

It is known since \cite{schatz} for conforming finite element discretization  and  it will be proved in this paper for nonconforming and
 for mixed finite element methods that  sufficiently fine triangulations allow for unique  discrete solution. One key argument in the proof is some representation
 formula for the
lowest-order Raviart-Thomas solution to \eqref{eq1} in terms of the Crouzeix-Raviart solution. This
circumvents the extra conditions on the coefficients from \cite{CHZ} to deduce the solvability
of the mixed finite element scheme and, thereby,  allows  a numerical analysis of the general linear indefinite problem at hand.
 The {\it a priori} error analysis shows a quasi-optimal error estimate
by best-approximation errors.

The robust {\it a posteriori} error control is feasible for sufficiently fine (although unstructured but shape-regular) meshes on the basis of some
{\it a priori} $L^2$ control for the nonconforming FEM by duality. This allows for reliable and efficient error estimates in terms of the explicit residual-based
error estimators up to generic constants and data approximation errors.

This paper is devoted to another approach to generalized saddle-point problems via an explicit equivalence to nonconforming finite element schemes
 for  general second-order linear indefinite and non-symmetric elliptic PDEs. The standard generalization of the Brezzi splitting lemma \cite{Brezzi}
 to more general possibly non-symmetric bilinear forms in \cite{CHZ} formulates various conditions on several boundedness and inf-sup constants.
 Those are essentially sufficient conditions and not equivalent to well-posedness. Observe that all conditions in \cite{CHZ} hold as well for some bilinear form which involves
 a homotopy parameter $\lambda$ which takes away the non-symmetry or indefiniteness for $\lambda=0$ and equals the bilinear form considered  in \cite{CHZ} for  $\lambda=1$.  For such a homotopy and certain critical values of $0< \lambda<1$,
the underlying PDE may have a zero eigenvalue, while the sufficient condition of \cite{CHZ} is convex in $\lambda$ and so holds for that critical value as well.
 This illustrates that we may encounter some general second-order linear PDE, where the conditions in \cite{CHZ} do not guarantee any
well-posedness of the continuous or the discrete situation, while the continuous problem is well-posed, and hence, some novel mathematical ideas are required
to ensure the solvability of the discrete solution in MFEM and their uniform boundedness {\it a~priori} for small meshes.

This paper assumes that the parameters in the general second-order linear elliptic PDE are
such that the associated boundary value problem is well-posed on the continuous level and shows with arguments like those in \cite{schatz}
 for the conforming case that there exists discrete solutions for a first-order nonconforming finite element method provided the mesh is sufficiently 
fine. Based on general conforming companions as part of the novel medius analysis, which utilizes mathematical arguments between {\it a~ priori} 
and {\it a~posteriori} analysis, this paper proves $L^2$ error and piecewise $H^1$ error estimates.

The remaining parts of the paper are organized as follows. Section 2 introduces  the weak and mixed weak formulations and equivalence of primal and mixed methods.
 Section 3 presents the Crouzeix-Raviart nonconforming finite
element methods (NCFEM) and discusses the solvability of the discrete problem and the related  {\it a~priori} and {\it a~posteriori}
 error estimates. Section 4 focuses  on Raviat-Thomas mixed finite element methods (RTFEM), the representation of RTFEM solution via NCFEM,
 and {\it a priori} error estimates for RTFEM. Section 5 establishes {\it a posteriori} error estimates for the discrete mixed  
formulation and  its  efficiency. Numerical experiments in  Section 6  concern to sensitivity  of  the {\it a~priori} and {\it a~postriori}
 error bounds and study  the performance of the related  adaptive algorithms.


This section  concludes with some notation  used through out this paper.
An inequality $A\lesssim B$ abbreviates $A\leq CB$, where $C>0$ is a mesh-size independent constant that depends only on the domain and 
the shape of finite elements;
 $A\approx B$ means $A \lesssim B \lesssim _{•} A$. Standard notation applies to Lebesgue and Sobolev spaces and $\norm{\cdot}$ abbreviates
 $\norm{\cdot}_{L^2(\Omega)}$ with   $L^2$ scalar product $(\cdot,\cdot)_{L^2(\Omega)}.$
Let $H^m(\Omega)$ denote the  Sobolev  spaces of order $m$ with  norm given by $\norm{\cdot}_m.$  
 The space of $\mathbb R^2$-valued  $L^2$ and $H^1$ functions defined over the domain $\Omega$ is denoted by
$L^2(\Omega; \mathbb R^2)$ and $H^1(\Omega; \mathbb R^2)$ respectively. Let $H(\text{div},\Omega)=\{{\bf q}\in L^2(\Omega; \mathbb R^2) : \: \text{div}~ {\bf q} \in L^2(\Omega) \} $ with the norm $~\norm{\cdot}_{H(\text{div},\Omega)}$ and 
its dual space  $H(\text{div},\Omega)^*.$\\

\section{ On Weak and Mixed Formulations}
This section introduces the minimal assumptions, the weak formulation with a reference to solvability, 
and the mixed formulation for the problem (\ref{eq1}) and their equivalence.
Define the  bilinear form
$a(\cdot,\cdot)$ for  $u,v \in H^1_0(\Omega)$ by 
$$a(u,v)=(\mathbf A \nabla u+u {\mathbf b}, \nabla v)_{L^2(\Omega)}+(\gamma \: u,v)_{L^2(\Omega)}.$$
The weak formulation of (\ref{eq1}) reads: Given $f\in {L^2(\Omega)},$ seek a function $u\in H^1_0(\Omega)$  such that
\begin{equation}\label{eq2}
 a(u,v)=(f,v)_{L^2(\Omega)} \qquad \mbox{for all}\; v\in H^1_0(\Omega).
\end{equation}
Throughout this paper, the following assumptions {\bf (A1)-(A2)} are posed on the coefficients and solution of the problem (\ref{eq1}). 
\begin{enumerate}
\item[{\bf (A1)}]  The coefficient matrix $\mathbf A \in L^{\infty}(\Omega;\mathbb R_{sym}^{2\times 2}) $ is  positive definite; that is,  there exist positive numbers
$\alpha$ and $\varLambda$ such that 
$\alpha |{\boldsymbol \xi}|^2 \leq \mathbf A(x){\boldsymbol \xi} \cdot {\boldsymbol \xi} \leq \varLambda |{\boldsymbol \xi}|^2 $ 
for  $ \text{a.e.} ~ x \in \Omega $ and  for all $ {\boldsymbol \xi} \in \mathbb R^2$. 
 Further, the coefficient matrix ${\bf A}$, vector ${\bf b}$ and $\gamma$ are Lipschitz continuous.
\item[{\bf (A2)}]  Given any $f \in L^2(\Omega)$, the problem (\ref{eq1}) has a unique  weak solution $u \in H^1_0 (\Omega).$ 
\end{enumerate}
The  dual problem reads:  Given $g\in L^2(\Omega)$, seek a solution $\Phi \in H^1_0(\Omega)$ such that 
\begin{equation}\label{adjoint-problem}
 a(v,\Phi)=(g,v)_{L^2(\Omega)} \qquad \mbox{for all}\; v\in H^1_0(\Omega).
\end{equation}
The unique solvability of (\ref{adjoint-problem}) follows by duality from the well-posedness of $ \mathcal L$, in {\bf (A2)} 
 and, as a consequence, $\norm{\Phi}_{1} \leq C \norm{g}.$
\begin{enumerate}
\item[{\bf (A3)}] Suppose that there exist some constants $0<\delta <1$ and $C(\delta)<\infty$ such that the unique solution $\Phi=\mathcal{L}^{-1}g$
of (\ref{adjoint-problem}) satisfies
  $\Phi \in H^{1+\delta}(\Omega)\cap H^1_0(\Omega)$ and
\begin{eqnarray}\label{adjoint-regularity}
\|\Phi\|_{1+\delta} \leq C(\delta) \|g\|.
\end{eqnarray}
\end{enumerate}

Since $0$ is not  part of the spectrum of  $\mathcal L$, the Fredholm alternative  \cite[Theorem 5 pp. 305-306]{evans} proves that 
the problem (\ref{eq1}) has a unique weak solution for each $f\in L^2(\Omega).$
For more detailed information on existence and uniqueness result of the  weak solution to (\ref{eq1}) or
 to (\ref{adjoint-problem}), see \cite[Theorem 8.3 pp. 181-182]{GT} or \cite[Theorem 4  pp. 303-305]{evans}.
For 
 (\ref{adjoint-regularity}), refer to \cite[cf. $\S$ 5.e and $\S$ 14.A] {dauge}. 

Introduce new variables ${\bf p}  =-({\mathbf A\nabla u+u {\mathbf b}})$ and ${\mathbf b}^*= \textbf A^{-1}{\bf b}$ and 
 rewrite (\ref{eq1}) as a  first-order system  
\begin{eqnarray}\label{eq3}
\begin{array}{lll} 
\mathbf A^{-1}{\bf p} +u {\mathbf b}^* + \nabla u =0  \; ~{\text{and}}~~
{\text{div}}~{\bf p}+\gamma \: u=f \;  {\rm in } \; \Omega. 
\end{array}
\end{eqnarray}
The mixed formulation  
seeks  $({\bf p},u)\in H(\text{div},\Omega) \times L^2(\Omega) $ such that 
\begin{eqnarray} \label{eq4}
\begin{array}{llll} 
(\textbf A^{-1}{ \bf p} + u {\bf b}^*,{\bf q})_{L^2(\Omega)}-({\rm div}~ {\bf q}, u)_{L^2(\Omega)}=0 \qquad \mbox {for all}\; {{\bf q}} \in H(\text{div},\Omega),\\
({\text{div}}~ {\bf p}, v )_{L^2(\Omega)}+(\gamma \: u,v)_{L^2(\Omega)}=(f,v)_{L^2(\Omega)} \qquad \qquad \mbox{for all}\; v \in L^2(\Omega).
\end{array}
\end{eqnarray}
\begin{thm} \label{lemma1.1} (Equivalence of primal and mixed formulation)
~The pair $({\bf p}, u)\in$ $ H(\text{\rm div},\Omega)\times L^2(\Omega)$ solves (\ref{eq4}) if and only if $u\in H^1_0(\Omega) $
 solves (\ref{eq1})~and \\
${{\bf p}}=-({\textbf A\nabla u+u {\mathbf b}}).$
\end{thm}
\textit{Proof.} Let $({\bf p},u)\in H(\text{div},\Omega)\times L^2(\Omega)$ solve  
(\ref{eq4}) and let $\phi \in {\cal D}(\Omega).$
Since  ${\bf q}:= \text{Curl} ~\phi:= (-\partial \phi/\partial x_2,  
\partial \phi/\partial x_1)$
is divergence-free and an admissible test function in  the  first  
equation of (\ref{eq4}), a formal integration by parts with $
\rm{curl}~  $ defined for any smooth vector field ${\bf r} =(r_1,r_2)$ by $
\rm{curl}~  {\bf r } := \partial r_1/\partial x_2 - \partial  
r_2/\partial x_1 $ proves
$$ \text{curl}~(\textbf A^{-1}{\bf p}+u {\bf b}^*)=0 \; \text{in} \;  
{\cal D}'(\Omega). $$
The Helmholtz decomposition shows for the simply-connected domain $\Omega$
that $ A^{-1}{\bf p}+u {\bf b}^*$ is the
gradient  of some  $v \in H^1_0(\Omega)$, namely;
\begin{equation*} \label{eqn6}
\textbf A^{-1}{\bf p}+u {\bf b}^*= \nabla v .
\end{equation*}
The substitution of this in the first equation of (\ref{eq4}) followed  
by an integration by parts
shows
\begin{eqnarray*}
 (\text{div}~ {\bf q}, v+u)_{L^2(\Omega)}=0 \qquad \text{for all  }{\bf q} \in  
H(\text{div},\Omega).
\end{eqnarray*}
It is known that the divergence operator $\text{div}:H(\text{div},
\Omega)\to L^2(\Omega)$
is surjective and so the preceding identity proves  $u+v=0$. (A direct  
proof follows with
the test function ${\bf q}= \nabla \psi$ for the solution
$\psi\in H^1_0(\Omega)$ of the Poisson problem
$  - \Delta \psi= u + v $ in $ \Omega$.)
This implies $u\in H^1_0(\Omega)$ and
\begin{equation}\label{eqn6.1}
\textbf A^{-1} {\bf p} +u {\bf b}^*=-\nabla u .
\end{equation}
This identity is recast into ${\bf p}=-(\textbf A \nabla u+ u {\bf b})
$ so that
the second equation of (\ref{eq4}) leads to  (\ref{eq1}).

\medskip

Conversely, let $u$ be a solution of (\ref{eq1}) and define
${\bf p}:=-(\textbf A \nabla u +u {\bf b)}\in L^2(\Omega;\mathbb{R}^2)$.
Then (\ref{eq1}) reads
\[
 \text{div }{\bf p} + \gamma \: u=f \quad\text{in }  {\cal  
D}'(\Omega).
\]
Since $f-\gamma \: u\in L^2(\Omega)$,  this implies ${\bf p}\in  
H(\text{div}, \Omega)$ and
the previous identity leads to
\[
 \text{div }{\bf p} + \gamma \: u=f \quad\text{a.e. in } \Omega.
\]
Now, an immediate consequence is  the second identity in  (\ref{eq4}).

The definition of ${\bf p}$ is equivalent to \eqref{eqn6.1}. The  
multiplication of
\eqref{eqn6.1} with any ${\bf q} \in H(\text{div}, \Omega)$ followed  
by an integration
over the domain $\Omega$ leads on the right-hand side to the  
$L^2(\Omega)$ product
of $-\nabla u$ and ${\bf q} $. That term allows for an integration by  
parts and so leads to
the first identity in  (\ref{eq4}).
This concludes the proof. \qed
The well-posedness of (\ref{eq1}) states that ${\mathcal {L}} :H^1_0(\Omega)\rightarrow H^{-1}(\Omega)$ is bounded and has a bounded inverse.
 This is an assumption on the coefficients which excludes zero eigenvalues in the Fredholm alternative, see \cite[Section 8.2]{GT}.
 The system (\ref{eq1}) is
equivalent to (\ref{eq4}) which implies that the operator 
\begin{eqnarray}\label{eqqq1} {\mathcal {M}} : \left\{
\begin{array}{l l}
 H(\text{div},\Omega)\times L^2(\Omega)\rightarrow H(\text{div},\Omega)^*\times L^2(\Omega),\\
({\bf q},v)\mapsto (\textbf A^{-1}{\bf q}+v \textbf b^*+\nabla v,~ \text{div}~{\bf q}+\gamma v)
 \end{array} \right.
\end{eqnarray}
has a range which includes $\{0\}\times L^2({\Omega})$; that is, for any $f\in L^2({\Omega})$ there exists $\mathcal M^{-1}(0,f),$
 which solves (\ref{eq4}) with the zero right-hand side in the first equation of (\ref{eq4}). The {\it a posteriori} error analysis relies on 
the well-posedness of the operator $\mathcal M$ even with a general right-hand side ${\bf g}\in H(\text{div},\Omega)^*$ in the first equation of (\ref{eq4}).

\begin{thm}\label{lemma2} (Well-posedness of mixed formulation)
~The linear operator $\mathcal M$ from (\ref{eqqq1}) is  bounded and has a bounded inverse. 
\end{thm}
{\it Proof.} The injectivity follows from that of $\mathcal L$ and the equivalence of (\ref{eq1}) and (\ref{eq4}) in Theorem \ref{lemma1.1} for ${\bf g}=0.$
The more delicate surjectivity follows in several steps. The step one is that for ${\bf g}=0$ and any $f \in L^2(\Omega)$, there exists some unique  $\mathcal M^{-1}(0,f)$ in (\ref{eqqq1}), because of the equivalence of (\ref{eq1}) and (\ref{eq4}).

In step two, let ${\bf g}=\nabla v$ be the gradient of some Sobolev function $v \in H^1_0(\Omega),$ i.e.,
\begin{eqnarray*}
 <{\bf g},{\bf q} >_{H(\text{div},\Omega)^*\times H(\text{div},\Omega)}&=& \int_{\Omega} \nabla v \cdot {\bf q} \; dx\\
& =&-\int_{\Omega}  v ~\text{div}~{\bf q}\; dx ~~\mbox{for all}\; ~{\bf q} \in H(\text{div},\Omega).
\end{eqnarray*}
Then, $\mathcal M ({\bf p},u)=({\bf g},f)$ is equivalent to 
$$
{\bf p}=\textbf A \nabla(v-u)-u\textbf b \;\; {\text {and }}\;\; \text{div}~ {\bf p}+\gamma u=f.
$$
The substitution of ${\bf p}$  in the second equation  shows

\[
  -\text{div}(\mathbf A\nabla u+u {\mathbf b})+ \gamma  u=f-\text{div} \, (\mathbf A\nabla v) \in H^{-1}(\Omega).
\]
Since  equation (\ref{eq1}) has a unique weak solution for a given right-hand side in $H^{-1}(\Omega)$ (from ({\bf A2}) and the Fredholm alternative),  the previous equation has  unique solution
$$u=\mathcal L^{-1}(f-\text{div} \, (\mathbf A\nabla v) \:)\in H^1_0(\Omega). $$
Since 
$$
{\bf p}:=\mathbf A\nabla (v-u)-u{\bf b}\in L^2(\Omega;\: \mathbb R^2)
$$
satisfies  $\text{div} ~{\bf p}=f- \gamma u\in L^2(\Omega),$  it follows ${\bf p} \in H(\text{div}, \Omega).$ Altogether, 
$$\mathcal M({\bf p}, u)= (\nabla v,f).$$
In step three, let ${\bf g}\in L^2(\Omega;\mathbb R^2)\subseteq H(\text{div},\Omega)^*$ and consider the Helmholtz decomposition of  ${\bf g}$ in the format
$$ \mathbf A {\bf g}=\mathbf A \nabla \alpha +\text{Curl}~\beta $$
for $\alpha \in H^1_0(\Omega)$ and $\beta \in H^1(\Omega) / \mathbb R$. This decomposition follows from the solution $\alpha$ of $-\text{div}(\mathbf A \nabla \alpha)=-\text{div}(\mathbf A {\bf g}) $ 
and the fact that the divergence free function $\mathbf A ({\bf g}-\nabla \alpha)$
equals a rotation in the simply-connected domain $\Omega.$

Since ${\bf g}=\nabla \alpha +\mathbf A^{-1}\text{Curl}~\beta$ and from step two, the superposition principle shows that  it remains to verify that
$$\mathcal M ({\bf p},u)= (\mathbf A^{-1} \text{Curl}~\beta,0 )$$
has a unique solution. Since $\text{div}\,(\text{Curl}\,\beta)=0$, this is equivalent to 
$$\mathcal M ({\bf p}-\text{Curl}~\beta, u)=0$$
with the obvious solution ${\bf p}=\text{Curl}~\beta \in H(\text{div},\Omega)$ and $u=0.$

In step four, let ${\bf g}=\nabla v$ for some $v\in L^2(\Omega)$ such that
$$ < {\bf g},{\bf q} >_{H(\text{div},\Omega)^*\times H(\text{div},\Omega)}
= -\int_{\Omega}  v ~\text{div}~{\bf q}\, dx ~~\mbox{for all} ~ {\bf q} \in H(\text{div},\Omega).
$$
This generalizes the step two in the sense that $v\in L^2(\Omega).$
The equation $\mathcal M ({\bf p},u)=(\nabla v,0)$ is equivalent to 
$$\mathcal M ({\bf p},u-v)=(-v\textbf b^*,-\gamma v).$$
This has a unique solution $({\bf p},u-v)$ in $ H(\text{div},\Omega)\times L^2(\Omega),$
because of step three (owing to $({\bf g},f)\in L^2(\Omega; \mathbb R^2 \times \mathbb R)$).

In step five, let $G\in H(\text{div},\Omega)^*$ with its Riesz representation $ {\bf {g}} \in H(\text{div},\Omega)$ in the Hilbert space $H(\text{div},\Omega)$, i.e.,
$$
\forall {\bf q} \in H(\text{div},\Omega) \qquad G({\bf q})=\int_{\Omega}({\bf g}\cdot{\bf q}+\text{div} ~{\bf g}~\text{div}~{\bf q})\;dx.
$$
Then, $\mathcal M ({\bf p}_1,u_1)=({\bf g},f)$ has a unique solution $({\bf p}_1,u_1)$ 
from step three and $\mathcal M ({\bf p}_2,u_2)=(-\nabla\text{div}~{\bf g},0)$
has a unique solution $({\bf p}_2,u_2)$ from step four with $v=\text{div}~{\bf g}\in L^2(\Omega).$
In conclusion, $({\bf p},u):=({\bf p}_1+{\bf p}_2,u_1+u_2)=\mathcal M^{-1}(G,f).$
This concludes the proof. \qed
\section{Non-Conforming Finite Element Methods}
This section describes the  Crouzeix-Raviart non-conforming finite element methods (NCFEM)  
for the problem \eqref{eq2}
and discusses {\it a priori } error estimates. 
\subsection{Regular Triangulation}
Let $\mathcal{T}$ be a regular triangulation  of the bounded  simply-connected polygonal Lipschitz domain 
${\Omega}\subset {\mathbb R}^2 $ into triangles such that 
$\cup_{T\in \mathcal{T}} T =\overline{\Omega}.$
Let $\mathcal{E} $ denote the set of all edges in $\mathcal{T}$, ${\mathcal{E}}({\partial \Omega})$ denote the 
set of all boundary edges in $\mathcal{T}$ and let $\mathcal{N}$ denote the set of vertices in   
$ \mathcal{T}.$
 Let $\text{mid}(E) $ denote the midpoint of the edge $E$ and $\text{mid}(T)$ denote the centroid of the triangle $T.$ 
 The set of edges of the element $T$ is denoted by $ {\mathcal{E}} (T).$ 
Let $h_T$ denote the diameter of the element $T \in \mathcal{T}$ and $h_{\mathcal T} \in P_0(\mathcal T)$ the piecewise constant mesh-size, 
$h_{\mathcal T}|_T:= h_T$ for all $T \in \mathcal T$ with  $h:= \max_{T\in \mathcal {T}} h_T.$  Let  
$|E|$ be the length of the edge $E\in \mathcal{E} $ with unit outward normal  $\nu_E.$\\
Let $\Pi_0$ be the $L^2$ projection onto  $P_0(\mathcal{T})$ and define $osc(f,\mathcal T):=\norm{h_{\mathcal T}(1-\Pi_0)f},$ where
$$
P_r(\mathcal{T})=\{v \in L^2(\Omega): \forall T\in \mathcal{T}, v|_T\in P_r(T) \}.
$$
Here and throughout this paper, $P_r(T)$, denotes the algebraic polynomials  of total degree at most $r \in \mathcal {N}$ as functions 
on the triangle $T\in \mathcal{T}.$ The $P_1$ conforming finite element space reads 
$$V(\mathcal{T}) := P_1(\mathcal{T}) \cap H^1_0(\Omega).$$ 
The jump of ${\bf q}$ across $E$ is denoted by $[{\bf q}]_E$; that is, for  two  neighboring triangles
$T_{+}$ and $T_{-},$
$$
[{\bf q}]_E(x) :=({\bf q}|_{T_{+}} (x)-{\bf q}|_{T_{-}}(x))~\text{  for } ~x \in E =\partial T_{+} \cap 
\partial T_{-}.
$$
The sign of $[{\bf q}]_E$ is defined by the convention that there is a fixed orientation of $\nu_E$ pointing 
outside of $T_{+}.$
Let $H^m(\mathcal{T})$ be the broken Sobolev  space of order $m$ with
broken Sobolev norm 
$$
\|\cdot\|_{H^m(\mathcal T)} := \left(\sum_{T\in \mathcal{T}} \|\cdot\|^2_{H^m({T})} \right)^{1/2}.
$$
%
The piecewise gradient $\nabla_{NC}: H^1(\mathcal {T})\longrightarrow L^2(\Omega; {\mathbb R}^2)$ acts as
$ \nabla_{NC}  {v}|_T = \nabla v|_T\;\;\;\text{for all} ~T \in \mathcal{T}.$ 
The broken Sobolev norm $\tnorm{\cdot}_{NC}$ abbreviates $(\mathbf{A} \nabla_{NC} ~\cdot,\nabla_{NC} ~\cdot)_{L^2(\Omega)}^{1/2}$ based on an 
underlying  triangulation $\mathcal{T}.$

\subsection{ Crouzeix-Raviart Non-Conforming Finite Element Methods }
This subsection defines  the non-conforming  finite element spaces and discusses the solvability of the discrete problem and the related {\it a priori} error estimates.

Given $ \displaystyle P_1(\mathcal T)$, the non-conforming Crouzeix-Raviart (CR) finite element space reads
\begin{eqnarray*}
&& CR^1(\mathcal T): =\{v \in P_1(\mathcal{T}): \forall E \in \mathcal E,~ v~ \text{ is continuous at mid($E$) }\},\\
&& CR^1_0(\mathcal T):=\{v \in CR^1(\mathcal T):v(\text{mid} (E))
=0 ~~\text{for all} ~E \in \mathcal E ({\partial\Omega}) \}.
\end{eqnarray*}
Let
\begin{eqnarray}\label{non2}
a_{NC} (w_{CR}, v_{CR}):=&\displaystyle \sum_{T\in \mathcal{T}}\int_{T} \Big( \left(
\mathbf{A} \nabla w_{CR} + w_{CR} {\bf b}\right)\cdot \nabla v_{CR} +
\gamma \, {w}_{CR} v_{CR}\Big) \;dx \nonumber\\
=& (\mathbf{A}\nabla_{NC}{w}_{CR}+ {w}_{CR}  {\bf b} ,\nabla_{NC} v_{CR})_{L^2(\Omega)}+(\gamma \, {w}_{CR},v_{CR})_{L^2(\Omega)}.
\end{eqnarray}
The  nonconforming finite element method  for (\ref{eq2}) seeks
 $ {u}_{CR}\in CR^1_0(\mathcal T)$ such that 
\begin{equation}\label{non1}
a_{NC} ( u_{CR}, v_{CR}) = (f, v_{CR})\;\;\;\;\;\text{for all} ~ v_{CR} \in CR^1_0(\mathcal T).
\end{equation}
Note that, 
$ a_{NC}(v,w) = a (v,w)\;\;\;\text {for}\; v,w\in H^1(\Omega).$
Observe that there are positive constants $\alpha_A$ and $M_A$ such that 
$$\alpha_A \|v\|^2_{H^1(\mathcal T)} \leq \tnorm{v}^2_{NC} \leq M_A \|v\|_{H^1(\mathcal T)}^2\;\;\;\;\text{for all}~
v\in  H^1_0(\Omega)+ CR^1_0(\mathcal T).
$$
The assumptions ({\bf{A1}}) implies that, the bilinear form $a_{NC}(\cdot,\cdot)$ satisfies the following properties (i)-(ii).\\
(i)  Boundedness. There exists a positive constant $M$ such
that 
\begin{equation} \label{non4}
| a_{NC}(v,w)| \leq M \tnorm{v}_{NC}\;\tnorm{w}_{NC}\;\;\;\;\;\text{for all}~ v, w \in 
H^1_0(\Omega) + CR^1_0(\mathcal T).
\end{equation}
(ii) G\r{a}rding-type inequality. There is a positive constant $\alpha$ and a nonnegative constant $~~~~~~~\beta$ such that 
\begin{equation} \label{non4-coer}
\alpha  \tnorm{v}^2_{NC} - \beta \|v\|^2  \leq a_{NC}(v,v) \;\;\;\;\;\text {for all}~ v \in H^1_0(\Omega) 
+ CR^1_0(\mathcal T).
\end{equation}
\subsection{Existence and Uniqueness of the Solution of NCFEM}
This subsection is devoted to a discussion on the unique solvability of the discrete problem  (\ref{non1}).
The conforming finite element approximation $\Phi_C \in V(\mathcal{T})$ to the problem (\ref{adjoint-problem})  seeks $\Phi_C \in V(\mathcal{T})$ with
\begin{equation}\label{conforming}
a(v_C, \Phi_C)=(g,v_C) \; \,{\rm for~all} \; \, v_C \in V({\mathcal T}).
\end{equation}
A simple modification of arguments given in \cite[Theorem 2]{schatz}  leads to the following error estimate.
Given any $\epsilon>0,$ there exists an $h_1 =h_1(\epsilon)>0$ such that 
for $0 <h\leq h_1$, if $\Phi\in H^1_0(\Omega)$ is a solution of (\ref{adjoint-problem}) 
and $\Phi_C \in V(\mathcal{T})$  satisfies (\ref{conforming}), then there holds
\begin{equation}\label{conforming-error-1}
\norm{\Phi-\Phi_C}\leq \epsilon  \norm{\Phi-\Phi_C}_1, 
\end{equation}
and since $g\in L^2(\Omega)$,
\begin{equation}\label{conforming-error-2}
\norm{\Phi-\Phi_C}_1\leq \epsilon  \norm{g}.
\end{equation}
The nonconforming finite element method (\ref{non1}) is well-posed even for  more general right-hand sides.
\begin{thm} (Stability) \label{thm2}
For sufficiently small maximum mesh size $h$ and for all $f_0\in L^2(\Omega)$ and ${\bf {f}}_1 \in L^2(\Omega;\mathbb R^2),$ the discrete problem
\begin{equation}\label{nc-modified}
a_{NC} ( u_{CR}, v_{CR}) = (f_0, v_{CR})+ ({\bf {f}}_1, \nabla_{NC} v_{CR})\;\;\;\;\;\text{for all} ~ v_{CR} \in CR^1_0(\mathcal T),
\end{equation}
has a unique solution $u_{CR}\in CR^1_0(\mathcal T)$.  Furthermore,
the solution is stable in the sense that
\begin{eqnarray} \label{stability}
\tnorm{u_{CR}}_{NC} \lesssim \norm{f_0}+ \norm{{\bf {f}}_1}.
\end{eqnarray}
\end{thm}
One of the key arguments in the proof of Theorem \ref{thm2} is the following consistency condition.
\begin{lm} (Consistency) \label{lemma1}
Let $\Phi$ be the unique solution of  (\ref{adjoint-problem}). For  $\epsilon > 0,$  
 there exists some $h_2 > 0$ such that for $0 < h \leq h_2$ it holds
\begin{eqnarray} \label{consistency}
\sup_{0\neq v_{CR} \in CR^1_0(\mathcal{T})} \frac {|a_{NC}(v_{CR},\Phi)- (g, v_{CR})_{L^2(\Omega)}|}
{ \tnorm{v_{CR}}_{NC}}  \leq \epsilon \|g\|\;\; \text{for all} \;  g\in L^2 (\Omega).
\end{eqnarray}
\end{lm}
{\bf Proof.} Given $v_{CR}\in CR_0^1(\mathcal{T})$, define a conforming approximation by
the averaging of the possible values (also known as the precise  
representation)
\begin{equation*}
 v_1(z):=v_{CR}^*(z):=\lim_{\delta\rightarrow 0} \frac{1}{|B(z,\delta)|} \int_{B(z,\delta)}  
v_{CR}dx
\end{equation*}
of the (possibly) discontinuous
$v_{CR}$ at any interior node $z\in \mathcal{N},$ where $B(z,\delta)$ is a ball of radius $\delta$ at $z.$ Linear  interpolation
of those values defines $v_1\in V({\mathcal T})$.
The second step defines $v_2\in P_2(\mathcal{T})\cap C_0(\Omega)$ which equals  
$v_1$ at all nodes ${\mathcal{N}}$ and satisfies
\begin{equation*}\label{nn1}
\int_E v_{CR}\,ds=\int_E v_2\, ds\quad\text{for all }E\in
{\mathcal{E}}.
\end{equation*}
The third step adds the cubic bubble-functions to $v_2$ such that
the resulting function $v_3\in P_3(\mathcal{T})\cap C_0(\Omega)$ equals $v_2$ along the edges and  
satisfies
\begin{equation}\label{nn2}
\int_T v_{CR}\, dx=\int_T v_3\,dx\quad\text{for all } T\in {\mathcal{T}}.
\end{equation}
An integration by parts shows
\begin{equation}\label{nn4}
 \int_T \nabla v_{CR}\,dx=\int_T\nabla v_3 \,dx\quad\text{for all }T\in {\mathcal{T}}.
\end{equation}
The approximation and stability properties of $v_3$ has been studied in former work
of preconditioners for nonconforming FEM \cite{Brennerenrichment}  (called enrichment therein).
This along with  standard arguments also proves  approximation properties and  
stability in the sense that
\begin{equation}\label{nnn5}
||  h_{\mathcal T}^{-1} (v_3-v_{CR})|| +    \tnorm{v_3}_{NC}    \leq C_1   \tnorm{v_{CR}}_{NC}.
\end{equation}
With (\ref{non2}), 
(\ref{adjoint-problem}), (\ref{nn2})-(\ref{nn4}) and   the definition of $\Pi_0,$ it follows that
\begin{eqnarray*}\label {non10}
a_{NC}(v_{CR},\Phi)&-&(g, v_{CR})_{L^2(\Omega)}\\
&=& ({{\mathbf A \nabla \Phi }}, \nabla_{NC} v_{CR} )_{L^2(\Omega)}+( {\bf b}\cdot \nabla \Phi+\gamma \Phi-g, v_{CR})_{L^2(\Omega)} \\
&=& (\Pi_0 {({{\mathbf A} \nabla \Phi )}}, \nabla  v_3)_{L^2(\Omega)}+( {\bf b}\cdot \nabla \Phi+\gamma \Phi- g, v_{CR})_{L^2(\Omega)}   \nonumber\\
&=& -((1-\Pi_0) {({\mathbf A \nabla \Phi })}, \nabla  v_3)_{L^2(\Omega)} \nonumber \\
&&+( (1-\Pi_0)({\bf b}\cdot \nabla \Phi+\gamma \Phi- g), v_{CR} -v_3)_{L^2(\Omega)}.\nonumber
\end{eqnarray*}
The Cauchy-Schwarz inequality with \eqref{nnn5} yields 
\begin{eqnarray*} \label{non11}
&& a_{NC}(v_{CR},\Phi)- (g, v_{CR})_{L^2(\Omega)}\nonumber\\
&&~~~\leq\|(1-\Pi_0) {({\mathbf A \nabla \Phi })}\| \: \|v_3\|_1 
+ C_1{osc}({g -\gamma \Phi-{\bf b}\cdot \nabla \Phi},\mathcal T ) \tnorm{v_{CR}}_{NC}. 
\end{eqnarray*}
This and the aforementioned stability~
$
\|v_3\|_1  \leq C_1  \tnorm{v_{CR}}_{NC}
$ ~
prove
\begin{eqnarray} \label{consistency3}
\sup_{0\neq v_{CR} \in CR^1_0(\mathcal T)}&& \frac {|a_{NC}(v_{CR},\Phi)- (g, v_{CR})_{L^2(\Omega)}|}{ \tnorm{v_{CR}}_{NC}}\nonumber \\
 && \leq  C_1\|(1-\Pi_0) {({\mathbf A \nabla \Phi })}\|
+ C_1 {osc}({g -\gamma \Phi-{\bf b}\cdot \nabla \Phi},\mathcal T ) .
\end{eqnarray}
The approximation property of $\Pi_0$ proves
that the first term on the right-hand side of (\ref{consistency3}) is bounded by
\begin{eqnarray} \label{p-estimate}
\|(1-\Pi_0) {({\mathbf A \nabla \Phi })}\| &\leq & 2 \norm{(1-\Pi_0){{\mathbf A}}}_{\infty} \norm{\nabla \Phi}
+\norm{{\mathbf A}}_{\infty} \norm{(1-\Pi_0) \nabla \Phi}\nonumber\\
&\leq & 2C \norm{(1-\Pi_0){{\mathbf A}}}_{\infty} \norm{g}+\norm{{\mathbf A}}_{\infty} \norm{(1-\Pi_0) \nabla \Phi}.
\end{eqnarray}
Given $\epsilon > 0,$ from (\ref{conforming-error-2}) there exists  $h_{3}=h_{3}(\epsilon)>0$ such that for $0< h \leq h_{3}$ 
$$
\norm{(1-\Pi_0) \nabla \Phi} \leq \norm{\Phi -\Phi_C}_1\leq \frac{ \epsilon}{4C_1\norm{\mathbf A}_\infty} \norm{g},
$$
 and 
 $\norm{(1-\Pi_0){\bf A}}_\infty \leq \frac{\epsilon}{8 C C_1} $. 
The boundedness of $\Phi \in H^1_0(\Omega)$ by $ \norm{g}$ shows
$$ osc(g-\gamma \Phi-{\bf b}\cdot \nabla \Phi,\mathcal T)\leq \norm{h (g-\gamma \Phi-{\bf b}\cdot \nabla \Phi) } \leq C_2 h \|g\|.$$
For $\epsilon >0$, there exists an $h_4> 0$ such that for $0<h<h_4,$ $ osc(g-\gamma \Phi-{\bf b}\cdot \nabla \Phi,\mathcal T)\leq \epsilon /2\norm{g}$.
Alltogether for $\epsilon >0$, there exists  $0 < h_2 \leq \min\{h_3, h_4\}$ such that (\ref{consistency}) holds. This concludes the proof.\qed
{\textit{ Proof of  Theorem \ref {thm2}.}}
The choice $v_{CR}=u_{CR}$ in (\ref{nc-modified}), the  G\r{a}rding's inequality (\ref{non4-coer}), and 
the discrete Friedrich inequality \cite[pp 301]{Brenner-Scott} $\norm{u_{CR}} \leq C_{dF} \tnorm{u_{CR}}_{NC}$ imply 
\begin{eqnarray} \label{nc13}
\alpha \tnorm{u_{CR}}_{NC}^2 \leq  \beta \norm{u_{CR}}^2 + \Big(C_{dF} \norm{f_0}+ \norm{{\bf {f}}_1}\Big) \tnorm{u_{CR}}_{NC},
\end{eqnarray}
Hence,
\begin{eqnarray} \label{nc14}
\tnorm{u_{CR}}_{NC} \leq  \frac{C_{dF}\beta}{\alpha} \norm{u_{CR}} + \frac{1}{\alpha}\Big(C_{dF} \norm{f_0}+ \norm{{\bf {f}}_1}\Big).
\end{eqnarray}
The Aubin-Nitsche duality argument allows for  an estimate of $ \|u_{CR}\|.$
Since $\mathcal{L}$ is an isomorphism, the dual problem (\ref{adjoint-problem}) has a unique solution   $\Phi\in H^1_0(\Omega),$  which satisfies
$\|\Phi\|_1\leq C \|g\|.$
The conforming finite element solution  $\Phi_C$ of (\ref{adjoint-problem})  satisfies (\ref{conforming}) for all $g\in L^2(\Omega).$
Since $V(\mathcal{T})\subset CR^1_0(\mathcal {T}),$  (\ref{nc-modified})  shows  for $v_{CR}=\Phi_C$ that
\begin{equation}\label{Phi}
a_{NC}(u_{CR},\Phi_C)= (f_0, \Phi_{C})+ ({\bf {f}}_1, \nabla_{NC} \Phi_{C}).
\end{equation}
Elementary algebra and (\ref{Phi}) show
\begin{eqnarray*} \label{non14}
(g,u_{CR})_{L^2(\Omega)} &=& a_{NC}(u_{CR},\Phi-\Phi_C) +(g,u_{CR})_{L^2(\Omega)}- a_{NC}(u_{CR},\Phi) \nonumber\\
&&+ (f_0, \Phi_{C})+ ({\bf {f}}_1, \nabla_{NC} \Phi_{C})\nonumber\\
& \leq & M \tnorm{u_{CR}}_{NC} \|\Phi-\Phi_C\|_1 + \Big(C_{dF} \norm{f_0} + \norm{{\bf {f}}_1}\Big) \norm{\Phi_{C}}_{1} \nonumber\\
&&+ \tnorm{u_{CR}}_{NC}
\sup_{0\neq v_{CR} \in CR^1_0(\mathcal T)} \frac {|a_{NC}( v_{CR},\Phi)- (g, v_{CR})_{L^2(\Omega)}|} {\tnorm{v_{CR}}_{NC}}.\nonumber
\end{eqnarray*}
For $\epsilon >0$, there exists an $h_5=h_5(\epsilon) >0$ such that the first term on the right-hand side is made
 $\leq \frac{\alpha }{2C_{dF}M\beta}\epsilon\tnorm{u_{CR}}_{NC}\norm{g}$ and from Lemma \ref{lemma1}, 
the third can be made $\leq \frac{\alpha }{2 C_{dF}\beta}\epsilon\tnorm{u_{CR}}_{NC}\norm{g}$ .
The choice of $g=u_{CR}$ proves
\begin{equation*} \label{L2-estimate}
\|u_{CR}\|  \leq \frac{\alpha \epsilon}{C_{dF} \beta}  \tnorm{u_{CR}}_{NC} +  C (C_{dF}  \norm{f_0} + \norm{{\bf {f}}_1}) .
\end{equation*}
For $0<\epsilon<1$, (\ref{nc14})   results in 
\begin{equation*} \label{H1-estimate}
 \tnorm{u_{CR}}_{NC}  \lesssim  \norm{f_0} + \norm{{\bf {f}}_1} .
\end{equation*}
This proves  the stability  estimate (\ref{stability}) under the assumption that (\ref{nc-modified}) has a solution.
The bound (\ref{stability}) implies also the uniqueness of solution of (\ref{nc-modified}). In fact, if the linear system of equations had
a non-trivial kernel, there would exist unbounded solutions in contradiction  to (\ref{stability}).\qed

\subsection{ {\it A Priori} Error Estimates for NCFEM}
This subsection discusses {\it a priori} error bounds for the  non-conforming finite
element solution. For related estimates, see \cite{Chen}. The following $L^2$ error control for 
nonconforming FEMs has been observed in \cite[Eq. (3.6)]{chenhoppe} but is left without a proof and stated under the restrictive assumption $\gamma\ge 0$.
\begin{thm}\label{theorem-ncm-error}($L^2$ and $H^1$ error)
~Let $u \in H^1_0(\Omega)$ be the unique weak solution of (\ref{eq2}),
 let $u_{CR}$ be the solution of (\ref{non1}).
 Then, for  $\epsilon >0$,  there exists sufficiently small mesh-size $h$ such that 
\begin{equation}\label{nc-error1}
\norm{u-u_{CR}} \leq \epsilon \tnorm{u-u_{CR}}_{NC}
\end{equation}
and for $f\in L^2(\Omega)$
\begin{equation}\label{nc-error0}
\tnorm{u-u_{CR}}_{NC} \leq  \epsilon  \|f\|. 
\end{equation}
\end{thm}
\textit{Proof.} 
The Aubin-Nitsche duality technique for  $g\in L^2(\Omega)$ plus (\ref{adjoint-problem}) and (\ref{conforming-error-2}) and some direct calculations prove, for any $v_C \in V(\mathcal T)$,  
that
\begin{align} \label{nc-error-4}
(g,&u-u_{CR})_{L^2(\Omega)} \nonumber\\
&= a_{NC}(u-u_{CR},\Phi-\Phi_C) +\Big( a_{NC}(u_{CR}- v_C, \Phi)-(g, u_{CR}-v_C)_{L^2(\Omega)}\Big)\nonumber\\
&  \leq M \tnorm{u-u_{CR}}_{NC}\; \|\Phi-\Phi_C\|_1 \nonumber \\
&~~+  \tnorm{u_{CR}-v_C}_{NC} \,
\sup_{0\neq w_{CR} \in CR^1_0(\mathcal{T})} \frac {|a_{NC} (w_{CR}, \Phi)- (g, w_{CR})_{L^2(\Omega)}|} { \tnorm{w_{CR}}_{NC}} \nonumber\\
& \leq \frac{\epsilon}{2} \tnorm{u-u_{CR}}_{NC}\; \|g\|  \nonumber \\
&~~+\inf_{v_C\in V(\mathcal{T})}  \tnorm{u_{CR}-v_C}_{NC}
 \sup_{0\neq w_{CR} \in CR^1_0(\mathcal{T})} \frac {|a_{NC}(w_{CR}, \Phi)- (g, w_{CR})_{L^2(\Omega)}|} { \tnorm{w_{CR}}_{NC}}. 
\end{align}
Since \cite{Carhop}
$$ \inf_{v_C\in V(\mathcal{T})}  \tnorm{u_{CR}-v_C}_{NC} \leq C_3 \tnorm{u-u_{CR}}_{NC}$$
for sufficiently small mesh size $h$, the consistency condition (\ref{consistency})  in (\ref{nc-error-4})  imply
\begin{eqnarray*} \label{nc-error-5}
(g,u-u_{CR})_{L^2(\Omega)} \leq \epsilon  \; \tnorm{u-u_{CR}}_{NC} \;\|g\|.
\end{eqnarray*}
Hence,
\begin{eqnarray} \label{nc-error-6}
\| u-u_{CR}\|= \sup_{0\neq g \in L^2(\Omega)}   \frac{|(g,u-u_{CR})_{L^2(\Omega)}|}{\|g\|} \leq  \epsilon \; \tnorm{u-u_{CR}}_{NC}.
\end{eqnarray}
This concludes the proof of (\ref{nc-error1}).\\
%
Given any $v_C \in V(\mathcal{T}) \subset CR_0^1(\mathcal T),$ the G\r{a}rding-type inequality (\ref{non4-coer}) shows
\begin{eqnarray*}\label{nc-error-1}
&&\alpha \tnorm{u_{CR}-v_C}^2_{NC} - \beta \|u_{CR}-v_C\|^2 \leq  a_{NC}(u_{CR}-v_C,u_{CR}-v_C) \\
 &&~~= a_{NC}(u-v_{C}, u_{CR}-v_C) + \Big( (f, u_{CR}-v_C)_{L^2(\Omega)} - a_{NC}(u,u_{CR}-v_C)\Big).\nonumber 
\end{eqnarray*} 
The discrete Friedrichs inequality $\|u_{CR}-v_C\| \leq C_{dF} \tnorm{u_{CR}-v_C}_{NC}$ leads to 
\begin{eqnarray*} \label{nc-error-2}
\alpha \tnorm{u_{CR}-v_C}_{NC} &\leq & 	 C_{dF} \beta\|u_{CR}-v_C\| +M  \|u- v_C\|_1 \\
&&+ \sup_{0\neq w_{CR} \in CR^1_0(\mathcal{T})} \frac {|a_{NC}(u, w_{CR})- (f, w_{CR})_{L^2(\Omega)}|}
{ \tnorm{w_{CR}}_{NC}}. \nonumber
\end{eqnarray*}
Write $u-u_{CR}:=( u- v_C) - ( u_{CR}- v_C)$ for an arbitrary
 $v_C$ in $V(\mathcal{T}).$  The preceding estimates plus triangle inequality show
\begin{eqnarray} \label{nc-error-3}
 \tnorm{u-u_{CR}}_{NC} & \leq & \frac{ C_{dF}\beta}{\alpha} \|u-u_{CR}\| + (\frac{ C_{dF}\beta}{\alpha}+1+\frac{M}{\alpha})\inf_{v_C \in V(\mathcal{T})} \|u- v_C\|_{1} \nonumber\\
&&+ \frac{1}{\alpha} \sup_{0\neq w_{CR} \in CR^1_0(\mathcal{T})} \frac {|a_{NC}(u, w_{CR})- (f, w_{CR})_{L^2(\Omega)}|}
{ \tnorm{w_{CR}}_{NC}}. 
\end{eqnarray}
The last term is controlled with Lemma \ref{lemma1} which remains valid for $u\in H^1_0(\Omega)$ and for all
$f\in L^2(\Omega).$ \\
%
The error analysis of  \cite[Theorem 2]{schatz}, shows for any $\epsilon >0,$ that 
there exists an $h_6=h_6(\epsilon)>0$ such that  for $0< h\leq h_6$, the conforming finite element solution $u_C \in V(\mathcal T)$  of 
(\ref{eq2}) satisfies
\begin{eqnarray} \label{new-nc-error}
\inf_{v_C \in V(\mathcal{T})} \|u- v_C\|_1 & \leq & \|u-u_C\|_1 \leq \epsilon \: \|f\|.
\end{eqnarray}
The combination of  (\ref{nc-error-6}), (\ref{new-nc-error}) and \eqref{consistency} 
 implies (\ref{nc-error0}) for sufficiently small $h$.  This concludes the proof. \qed

\subsection{{\it A Posteriori} Error  Analysis for NCFEM}
This subsection is devoted to {\it a posteriori} error analysis of NCFEM with the residual
\begin{equation}\label{nc-residual}
{\mathcal Res}_{NC}(w):=(f,w)_{L^2(\Omega)}-a_{NC}(u_{CR},w)\qquad \text{for all } w\in V +CR^1_0(\mathcal{T}).
\end{equation}%
\begin{thm} ({\it A~posteriori} error control)
~Provided the mesh-size is sufficiently small, it  holds
 \begin{equation} \label{nc-estimator}
 \tnorm{u-u_{CR}}_{NC}\lesssim \norm{{\mathcal Res}_{NC}}_{H^{-1}(\Omega)}+\min_{v\in V} \tnorm{u_{CR}-v}_{NC}.
\end{equation}
\end{thm}
{\it Proof}.  The proof utilizes the nonconforming interpolant  $I_{NC}:H^1(\Omega)\rightarrow CR^1(\mathcal T )$ defined by 
\[ I_{NC}v(\text{mid}(E)) := \frac{1}{|E|}\int_E v~ds ~~~  \text{for all} ~v \in H^1(\Omega). \]
The  G\r{a}rding's inequality (\ref{non4-coer}) for $ e:= u-u_{CR}$ plus elementary algebra with the bilinear forms $a$ and $a_{NC}$ plus (\ref{eq2})
for $v:=u-v_4$ with $v_4\in V$ and (\ref{non1}) for $v_{CR}:= I_{NC}u-u_{CR}$ shows that 
$w:=u-v_4+ u_{CR}-I_{NC}u$ satisfies
\begin{equation}\label{aaaa}
 \alpha\tnorm{e}^2_{NC} -\beta\norm{e}^2 \leq (f,w)_{L^2(\Omega)}-a_{NC}(u_{CR},w)+a_{NC}(e,v_4-u_{CR}).
\end{equation}
Given $v_{CR}$, design $v_4\in P_4(\mathcal{T})\cap C_0(\Omega)\subseteq V$ with
\[ \forall~ p\in P_0(\mathcal T) ~~\int_\Omega \nabla v_4 \cdot p~ dx =\int_\Omega \nabla v_{CR} \cdot p~ dx,\]
\[ \forall ~w\in P_1(\mathcal T) ~~\int_\Omega  v_4 \cdot w ~dx=\int_\Omega  v_{CR} \cdot w~ dx.\]
The choice of the $P_4$-conforming companion $v_4\in P_4(\mathcal{T})\cap C_0(\Omega)$ 
with $I_{NC}v_4=v_{CR}$ allows for $C_{apx}\approx 1$ with
\begin{equation}\label{apost1}
 \tnorm{u_{CR}-v_4}_{NC}\leq C_{apx} \min_{v\in V} \tnorm{u_{CR}-v}_{NC}.
\end{equation}
The proof of (\ref{apost1}) follows from the  analogous arguments for $v_3$ in Lemma \ref{lemma1}.
 (\ref{aaaa}) shows
\begin{equation}\label{apost2}
 \tnorm{e}^2_{NC}\leq \frac{\beta}{\alpha}\norm{e}^2 +\frac{1}{\alpha} {\mathcal Res}_{NC}(w)+\frac{M}{\alpha} \tnorm{e}_{NC}
 \tnorm{u_{CR}-v_4}_{NC}
\end{equation}
with the nonconforming residual ${\mathcal Res}_{NC}(w)$ of (\ref{nc-residual}).
Note that (\ref{non1}) implies
\begin{equation}\label{apost3}
 P_1(\mathcal{T})\cap C_0(\Omega)\subseteq CR^1_0(\mathcal{T})\subseteq {\mathcal Ker} {\mathcal Res}_{NC}.
\end{equation}
The dual norm and  triangle inequality imply
\begin{equation*}
 {\mathcal Res}_{NC}(w)={\mathcal Res}_{NC}(u-v_4)\leq \tnorm{{\mathcal Res}_{NC}}_{H^{-1}(\Omega)}(\tnorm{e}_{NC}+\tnorm{u_{CR}-v_4}_{NC}).
\end{equation*}
This and (\ref{apost2}) prove
\begin{equation*}
 \tnorm{e}^2_{NC}\leq \frac{2\beta}{\alpha}\norm{e}^2+\frac{3}{\alpha^2}\tnorm{{\mathcal Res}_{NC}}^2_{H^{-1}(\Omega)}+\big(\frac{2M^2}{\alpha^2}+{1}\big) \tnorm{u_{CR}-v_4}^2_{NC}.
\end{equation*}
Theorem \ref{theorem-ncm-error} shows ${\norm{e}} \leq \frac{\alpha \epsilon}{2\beta }\tnorm{e}_{NC}$  and hence,  for $\epsilon > 0$ with $0<\epsilon<1$, 
 there exists a sufficiently small mesh-size $\norm{h_{\mathcal T}}_{L^\infty (\Omega)}<<1$ such that (\ref{apost1}) shows
\begin{equation*}
 \tnorm{e}^2_{NC}\leq \frac{3}{\alpha^2}\tnorm{{\mathcal Res}_{NC}}^2_{H^{-1}(\Omega)}+C_{apx}^2\big(\frac{2M^2}{\alpha^2}+{1}\big)\min_{v\in V} \tnorm{u_{CR}-v}^2_{NC}.
\end{equation*}
This implies (\ref{nc-estimator}) and 
 concludes the proof.\qed
 The analysis of the residual ${{\mathcal Res}_{NC}}\in {H^{-1}(\Omega)}$ with the kernel property (\ref{apost3}) is by now standard \cite{Car1,CH}.
 With ${\bf p}_{CR}:= -(\mathbf{A}\nabla_{NC}u_{CR}+u_{CR}\mathbf b),$ the explicit residual-based error estimator of \cite{Car1} reads
\begin{equation}
 \eta(\mathcal{T}):=\norm{h_{\mathcal T}(f-\gamma u_{CR}-\text{div}_{NC} {\bf p}_{CR})}+\norm{h_E^{1/2} [{\bf p}_{CR}]_E\cdot\nu_E }_{L^2(\cup E)}.
\end{equation}
Further details are, therefore, omitted. The residual $\min_{v\in V} \tnorm{u_{CR}-v}_{NC}$ is easily estimated by $v_4.$
\begin{remk}
 The  general {\it a~posteriori} error control can be contrasted with  \cite[Theorem 3.1]{chenhoppe} for $\gamma\ge 0$, 
where normal jumps arise  which do not play any role in this paper. 
\end{remk}
\section{Mixed Finite Element Methods}
This section discusses the lowest-order Raviart-Thomas  mixed finite element formulation and its equivalence to the NCFEM solution and derives
 {\it a priori} error estimates for the mixed method.
\subsection{Raviart-Thomas Finite Element Methods (RTFEM)}
With respect to the shape-regular triangulation $\mathcal T,$ the lowest-order Raviart-Thomas space reads 
\begin{align*}
RT_0(\mathcal{T}):=\{{\bf q}\in H(\text{div},\Omega): & \: \forall T \in \mathcal{T} 
~\exists {\bf c} \in \mathbb{R}^2~ \exists d \in \mathbb{R}~ \; \forall {\bf x} \in T, ~{\bf q}({\bf x})={\bf c}+ d ~ {\bf x}\\
&\text{ and} \; ~\forall E\in \mathcal{E}(\Omega), [{\bf q}]_{E} \cdot \nu_E=0\} .  
\end{align*}
Throughout this paper, $\mathbf A_h:= \Pi_0 \mathbf A$, $\mathbf b_h:= \Pi_0 \mathbf b$,  ${\bf b}^*_h:= \mathbf A_h^{-1}{\bf b}_h$, $\gamma_h:= \Pi_0 \gamma$, and $f_h:=\Pi_0 f$ denote the respective piecewise constant
 approximations of $\mathbf A$,$~{\bf b},$~${\bf b}^*$,~$\gamma$ and $f$. 
The discrete mixed finite element problem (RTFEM) for \eqref{eq4} seeks 
$ ({\bf p_M}, u_M) \in RT_0(\mathcal {T})\times P_0(\mathcal T)$ with
\begin{eqnarray}
(\mathbf{A}_h^{-1}{\bf p_M}+ u_M {\bf b}_h^*,{\bf q} _{RT})_{L^2(\Omega)}-(\text{div}~{\bf q}_{RT},u_M)_{L^2(\Omega)}=0
 ~~\text{for all}\, {\bf q}_{RT}\in RT_0(\mathcal {T}), \label{eqna1}\\
(\text{div}~ {\bf p_M},v_h)_{L^2(\Omega)}+(\gamma_h u_M, v_h)_{L^2(\Omega)}=(f_h,v_h)_{L^2(\Omega)} ~ \text {for all} ~ v_h \in P_0(\mathcal{T}).~~~ \label{eqna2}
\end{eqnarray}
\subsection{Equivalence of RTFEM and NCFEM}
The piecewise constant approximations  ${\mathbf A}_h$ and ${\bf b_h}$ of 
$ \mathbf{A} $ and ${\bf b}$ and
\begin{eqnarray} 
&& \displaystyle \tilde u_M({\bf x})=\left(1+\frac{S(T)}{4} \gamma_h \right)^{-1} \left(\Pi_0 \tilde{u}_{CR}+\frac{S(T)}{4}f_h \right)  \;\; \text{for}~ x \in T\in \mathcal T,\label{umt}\\
&& ~~S(T)= \displaystyle{\int_T} ({\bf x}-\text{mid}(T))\cdot \mathbf{A}^{-1}_h({\bf x}-\text{mid}(T)) \,d{\bf x}\;\; \text{for} ~ T\in \mathcal T,\label{s}
\end{eqnarray}
define a modified nonconforming FEM problem
\begin{eqnarray}\label{eqna3}
 (\mathbf{A}_h\nabla_{NC} \tilde{u}_{CR}+\tilde{u}_M {\bf b}_h,\nabla_{NC} v_{CR})&&+(\gamma_h\tilde{u}_M,v_{CR})\nonumber\\
&&= (f_h,v_{CR})~~~\text{ for all }\, v_{CR}\in CR^1_0(\mathcal T).
\end{eqnarray}
\begin{thm}
\label{existence-mncm}(Stability)
For sufficiently  small mesh-size $h$, there exists a unique solution $\tilde{u}_{CR}\in CR^1_0(\mathcal T)$ to discrete problem (\ref{eqna3}) with 
\begin{equation}\label{tilde-u-cr}
\tnorm{\tilde{u}_{CR}}_{NC}\lesssim \norm{{f}_h}.
\end{equation}
\end{thm}
\textit{Proof.}  A substitution of  $ \tilde u_M $ in (\ref{eqna3})  leads to
\begin{equation}\label{exi11}
 \tilde a_{NC}(\tilde u_{CR},v_{CR})=(\tilde f_h, v_{CR}) ~~\text{for all} ~v_{CR} \in CR^1_0(\mathcal{T})\qquad
\end{equation}
with $S(\mathcal{T})|_T =S(T)$ and
\begin{eqnarray*}
\tilde a_{NC}(\tilde u_{CR},v_{CR})&:=& (\mathbf{A}_h\nabla_{NC} \tilde{u}_{CR}+{\bf b}_h (1+ \frac{S(\mathcal{T})}{4} \gamma_h)^{-1}  (\Pi_0\tilde{u}_{CR}) ,\nabla_{NC} v_{CR})_{L^2(\Omega)}\\
&&+~(\gamma_h(1+ \frac{S(\mathcal{T})}{4} \gamma_h)^{-1} (\Pi_0\tilde{u}_{CR}), v_{CR})_{L^2(\Omega)}, \\
(\tilde f_h, v_{CR})_{L^2(\Omega)}&:=&(f_h,v_{CR})_{L^2(\Omega)}-({\bf b}_h (1+ \frac{S(\mathcal{T})}{4} \gamma_h)^{-1} \frac{S(\mathcal{T})}{4} f_h,\nabla_{NC} v_{CR})_{L^2(\Omega)}\\
&&~-(\gamma_h(1+\frac{S(\mathcal{T})}{4} \gamma_h)^{-1} \frac{S(\mathcal{T})}{4} f_h,v_{CR})_{L^2(\Omega)}.
\end{eqnarray*}

The stiffness matrix related to (\ref{exi11}) is very similar to  that of (\ref{non1}) except for some data perturbation and 
the substitution of $\Pi_0\tilde{u}_{CR}$ instead of ${u}_{CR}$  in two lower-order terms.  The last substitution models one-point integration, 
and since the variable $\tilde{u}_{CR}$ is controlled in the energy norm $\tnorm{\cdot}_{NC},$ it acts as some perturbation as well.
All these perturbations tends to zero as the maximal mesh-size tends to zero and hence, the existence, uniqueness and stability results may be deduced as in Subsection 3.3.

To be more specific, 
the choice $v_{CR}=\tilde u_{CR}$ in (\ref{exi11}) implies
\begin{equation}\label{nn3}
\tnorm{\tilde u_{CR}}_{NC} \lesssim \; \|\tilde {u}_{CR}\| + \|\tilde{f}_h\|.
\end{equation}
The Aubin-Nitsche duality argument allows for  an estimate of $ \|\tilde u_{CR}\|$.  Recall that  for given $g\in L^2(\Omega)$, $\Phi\in H^1_0(\Omega)$ is the unique solution of the dual problem $a(v,\Phi)=(g,v)$  from Subsection 3.3
and the conforming finite element solution $\Phi_C$ 
of (\ref{conforming})  satisfies the estimate (\ref{conforming-error-2}).

Since $V(\mathcal{T})\subset CR^1_0(\mathcal {T}),$ the choice of $v_{CR}=\Phi_C$ in (\ref{exi11}) yields 
\begin{equation}\label{Phi-C}
\tilde{a}_{NC}(\tilde u_{CR},\Phi_C)= (\tilde f_h, \Phi_{C}).
\end{equation}
An elementary algebra with (\ref{Phi-C}) and the discrete  Friedrich inequality shows
\begin{eqnarray} \label{nonc14}
(g,\tilde u_{CR})_{L^2(\Omega)} &=& \tilde a_{NC}(\tilde u_{CR},\Phi-\Phi_C)+(\tilde f_h, \Phi_{C})
 +(g,\tilde u_{CR})_{L^2(\Omega)}- \tilde a_{NC}(\tilde u_{CR},\Phi)\nonumber\\
& \lesssim &  \tnorm{\tilde u_{CR}}_{NC} \|\Phi-\Phi_C\|_1 + \|\tilde{f}_h\|\; \|\Phi_C\|_{1}\nonumber\\
&&+ \tnorm{\tilde u_{CR}}_{NC}
\sup_{0\neq v_{CR} \in CR^1_0(\mathcal T)} \frac {|\tilde a_{NC}( v_{CR},\Phi)- (g, v_{CR})|} {\tnorm{v_{CR}}_{NC}}.
\end{eqnarray}
The last term on the right-hand side of (\ref{nonc14}) is 
\begin{eqnarray*}
 &&\tilde a_{NC}(v_{CR},\Phi)-(g, v_{CR})_{L^2(\Omega)}\\
&&= a_{NC}(v_{CR},\Phi)-(g, v_{CR})_{L^2(\Omega)}-( \nabla _{NC} v_{CR},({\mathbf A- \mathbf A_h}) \nabla \Phi)_{L^2(\Omega)}\\
&&~~~-(v_{CR},(\mathbf b-\mathbf b_h)\cdot\nabla \Phi+(\gamma-\gamma_h) \Phi)_{L^2(\Omega)}-(v_{CR}-\Pi_0 v_{CR},\mathbf b_h\cdot\nabla \Phi+ \gamma_h \Phi)_{L^2(\Omega)}\\
 &&~~~-\big (\frac{S(\mathcal{T})}{4}\gamma_h (1+ \frac{S(\mathcal{T})}{4}\gamma_h)^{-1}\Pi_0 v_{CR},\mathbf b_h\cdot\nabla \Phi+\gamma_h \Phi\big )_{L^2(\Omega)}.\nonumber
 \end{eqnarray*}
The Cauchy-Schwarz inequality, the approximation property of $\Pi_0$ and $S(T)\approx h^2$ lead to 
\begin{eqnarray*} \label{consistency2}
\sup_{0\neq v_{CR} \in CR^1_0(\mathcal T)} &&\frac {|\tilde{a}_{NC}( v_{CR},\Phi)- (g, v_{CR})_{L^2(\Omega)}|}{ \tnorm{v_{CR}}_{NC}}\nonumber \\
&&\lesssim\sup_{0\neq v_{CR} \in CR^1_0(\mathcal T)} \frac {|{a}_{NC}( v_{CR}, \Phi)- (g, v_{CR})_{L^2(\Omega)}|}{ \tnorm{v_{CR}}_{NC}}\qquad\nonumber \\
&&~~~+ \big(h+\norm{{\bf A-A_h}}_\infty+\norm{{\bf b-b_h}}_\infty+\norm{{ \gamma-\gamma_h}}_\infty \big )\norm{\Phi}_1.
\end{eqnarray*}
Lemma \ref{lemma1}, $\norm{{\bf A-A_h}}_\infty\leq \epsilon,~\norm{{\bf b-b_h}}_\infty\leq \epsilon,~\norm{{ \gamma-\gamma_h}}_\infty\leq \epsilon $ 
for  $\epsilon>0$ and $\norm{\Phi}_1\leq C \norm{g}$ result in
\begin{eqnarray} \label{consi_2}
\sup_{0\neq v_{CR} \in CR^1_0(\mathcal T)} &&\frac {|\tilde{a}_{NC}( v_{CR},\Phi)- (g, v_{CR})_{L^2(\Omega)}|}{ \tnorm{v_{CR}}_{NC}}\lesssim \epsilon \norm{g}.
\end{eqnarray}
The combination with (\ref{conforming-error-2}) and  (\ref{nonc14})-(\ref{consi_2}) leads to 
$(g,\tilde u_{CR})  \lesssim (\epsilon \; \tnorm{\tilde u_{CR}}_{NC}+\|\tilde{f}_h\|)\;\|g\|.$
Hence, the boundedness of $\|\tilde{f}_h\| \lesssim \|f_h\|$  yields
$$\|\tilde u_{CR}\|  \lesssim \epsilon \; \tnorm{\tilde u_{CR}}_{NC}+\|{f}_h\|.$$
A substitution   in (\ref{nn3})  for sufficiently small $h$ results in
\begin{equation*} 
 \tnorm{\tilde u_{CR}}_{NC}  \lesssim \|{f}_h\|.
 \end{equation*}
Since $f_h=0$ shows that $\tilde {u}_{CR}=0,$ uniqueness follows. This also  implies existence of the discrete solution.
\qed
\begin{thm} \label{mfem}
(Equivalence of RTFEM and NCFEM)
~Recall $\tilde u_M $ and $S(T)$ from \eqref{umt}-\eqref{s} and  let $ \tilde{u}_{CR}\in CR^1_0(\mathcal T)$ solve (\ref{eqna3}). Then
\begin{equation} \label{pmt}
 \tilde{\bf p}_M({\bf x})=-\left(\mathbf{A}_h\nabla _{NC} \tilde{u}_{CR}+\tilde u_M{\bf b}_h \right)+\left(f_h-\gamma_h \tilde u_M \right)
 \frac{\left({\bf x}-\text{\rm mid}(T)\right)}{2} \text{ for}~ {\bf x}\in T \in \mathcal{T}
\end{equation}
defines $ \tilde{\bf p}_M \in RT_0(\mathcal {T}) \subset H({\rm div},\Omega)$ and 
the pair $ (\tilde {\bf p}_M, \tilde u_M)$ satisfies (\ref{eqna1})-(\ref{eqna2}).
Conversely, for any solution $(\tilde{\bf p}_M, \tilde u_M)$ in $ RT_0(\mathcal T) \times P_0(\mathcal T)$ of (\ref{eqna1})-(\ref{eqna2}) the solution
 $\tilde u_{CR} \in CR^1_0(\mathcal T)$ of (\ref{eqna3}) satisfies (\ref{umt}) and (\ref{pmt}).
\end{thm}
\textit{Proof.} Note that the continuity of the normal components on the boundaries of the triangles 
$T \in {\mathcal T}$ reflects the conformity $ RT_0(\mathcal {T}) \subset H(\text{div},\Omega).$
Given an interior edge  $E$  shared by neighboring triangles $T_+, T_- \in \mathcal{T}$ with unit normal $\nu_E$ pointing from $T_-$ to $T_+$,
let $\psi_E$ denote the non-conforming basis function  defined on an interior edge such that  
$\psi_E (\text{mid}(E))=1,$ while $~\psi_E(\text{mid}(F))=0 ~\text{for all } F \in \mathcal{E}\setminus\{E\}. $ 
A piecewise integration by parts shows 
\begin{align}\label{rep1}
 (\tilde{\bf p}_M,&\nabla_{NC} \psi_E)_{L^2(\Omega)}+(\text{div}_{NC} \: \tilde{\bf p}_M,\psi_E)_{L^2(\Omega)} =
 \int_{\partial T_+ \cup \partial T_- } \tilde{\bf p}_M \cdot \nu \psi_E \,ds\nonumber\\
&=\int_{E}(\tilde{\bf p}_M|_{T_+} \cdot \nu|_{T_+} +\tilde{\bf p}_M|_{T_-} \cdot \nu|_{T_-} )\psi_E \,ds
= |E|[\tilde{\bf p}_M]\cdot \nu_E,
\end{align}
where $\text{div}_{NC}v|_T=\text{div}~v|_T$.  
The definition of $\tilde{\bf p}_M$, (\ref{eqna3}) and the fact  
$$\left((f_h-\gamma_h \tilde u_M)\left({\bf x}-\text{mid}(T)\right)/2, \nabla_{NC} \psi_E \right)_{L^2(\Omega)} =0,$$ 
imply
$$ (\tilde{\bf p}_M,\nabla_{NC} \psi_E)_{L^2(\Omega)} +(\text{div}_{NC} \: \tilde{\bf p}_M,\psi_E)_{L^2(\Omega)}=0.$$
Hence, (\ref{rep1}) shows $|E|[\tilde{\bf p}_M]\cdot \nu=0.$ Since the edge $E$ is arbitrary in 
$\mathcal E(\Omega)$,  $\tilde{\bf p}_M \in RT_0(\mathcal {T}) \subset H(\text{div},\Omega).$
Since the distributional divergence is the piecewise one, (\ref{pmt}) proves  
$\text{div}_{NC}~ \tilde{\bf p}_M({\bf x})=f_h-\gamma_h \tilde u_M$. 
 Hence, (\ref{eqna2}) is satisfied.
A use of the definition of $\Pi_0$, an application of element-wise integration by parts, some elementary properties of elements in $RT_0(\mathcal {T}), \;
 CR^1_0(\mathcal {T})$, and \eqref{pmt} yield
\begin{eqnarray*}
(\mathbf{A}^{-1}_h\tilde{\bf p}_M &+&\tilde u_M{\bf b}^*_h ,~{\bf q}_{RT})_{L^2(\Omega)}-(\text{div}~{\bf q}_{RT},~\Pi_0\tilde{u}_{CR})_{L^2(\Omega)}\\
&=& (\mathbf{A}^{-1}_h\tilde{\bf p}_M + \tilde u_M{\bf b}^*_h ,~{\bf q}_{RT})_{L^2(\Omega)}-(\text{div}~{\bf q}_{RT},~\tilde{u}_{CR})_{L^2(\Omega)} \nonumber \\
&=& (\mathbf{A}^{-1}_h\tilde{\bf p}_M+\tilde u_M {\bf b}^*_h , ~{\bf q}_{RT})_{L^2(\Omega)}+(\nabla_{NC} \tilde{u}_{CR},~{\bf q}_{RT})_{L^2(\Omega)} \nonumber \\
&=& (\mathbf{A}^{-1}_h(f_h-\gamma_h \tilde u_M)(\bullet-\text{mid}(\mathcal {T}))/2,~{\bf q}_{RT})_{L^2(\Omega)}. \nonumber 
\end{eqnarray*}
Recall $S(\mathcal {T})|_T=S(T)$ and the definition of $S(T)$ from \eqref{s}. Some algebraic calculations with ${\bf q}_{RT} \in  RT_0(\mathcal {T}) $ and 
$\displaystyle \int_T ({\bf x}-\text{mid}(T)) \: dx =0$ yield
\begin{align*}
 (\mathbf{A}^{-1}_h\tilde{\bf p}_M &+\tilde u_M{\bf b}^*_h ,~{\bf q}_{RT})_{L^2(\Omega)}-(\text{div}~{\bf q}_{RT},\Pi_0\tilde{u}_{CR})_{L^2(\Omega)}\\
&=((f_h-\gamma_h \tilde u_M)\mathbf{A}^{-1}_h(\bullet-\text{mid}(\mathcal {T}))/2),~(\bullet-\text{mid}(\mathcal {T}))/2 \; \text{div}~{\bf q}_{RT})_{L^2(\Omega)}\\
&=\left(\frac{S(\mathcal {T})}{4}(f_h-\gamma_h \tilde u_M), \text{div}~{\bf q}_{RT} \right)_{L^2(\Omega)}.
\end{align*}
An appropriate re-arrangement shows  
that the pair $ (\tilde{\bf p}_M, \tilde u_M )$ satisfies (\ref{eqna1}).
This concludes the proof of the first part. 
\\
To prove the converse implication, let $(\tilde{\bf p}_M, \tilde u_M)$ in $ RT_0(\mathcal T) \times P_0(\mathcal T)$ be some 
solution of (\ref{eqna1})-(\ref{eqna2}). The discrete Helmholtz decomposition \cite{arnold} states for the simply-connected domain $\Omega$ that
the piecewise constant vector function $-\Pi_0(\mathbf A^{-1}_h	\tilde {\bf p}_M+\tilde u_M{{\bf b}^*_h})\in P_0(\mathcal T;\mathbb R^2)$ equals a discrete gradient 
$\nabla_{NC} \alpha_{CR}$ of some nonconforming function $\alpha_{CR}\in CR^1_0(\mathcal T)$ plus the Curl~$\beta_c$ of some piecewise affine conforming function 
$\beta_c\in P_1(\mathcal T)\cap C(\bar{\Omega});$ that is,
\[
  -(\Pi_0\mathbf A^{-1}_h	\tilde {\bf p}_M+\tilde u_M{{\bf b}^*_h})=\nabla_{NC} \alpha_{CR} +{\text {Curl}~}\beta_c.
\]
The argument to verify this is to define $\alpha_{CR}$ as the solution of a Poisson problem of  a nonconforming FEM with the right-hand 
side $-(\tilde {\bf p}_M+\tilde u_M{\bf b}_h,\mathbf A^{-1}_h\nabla_{NC} v_{CR})_{L^2(\Omega)}$ as a functional in $ v_{CR}\in CR^1_0(\mathcal T)$.
Once $\alpha_{CR}$ is determined, the difference $\nabla_{NC}\alpha_{CR}+\Pi_0 \mathbf A^{-1}_h	\tilde {\bf p}_M+\tilde u_M{{\bf b}^*_h}$ is $L^2(\Omega)$ orthogonal 
onto $\nabla_{NC} CR^1_0(\mathcal T).$ Hence, it equals the Curl of some Sobolev functions so that 
Curl~$\beta_c:= (-\frac{\partial \beta_c}{\partial x_2},\frac{\partial \beta_c}{\partial x_1})$ 
is piecewise  constant. This concludes the proof of the above discrete Helmholtz decomposition.\\
Since Curl~$\beta_c=:{\bf q}_{RT}$ is  a divergence free Raviart-Thomas function, (\ref{eqna1}) implies
\[ \norm{{\text {Curl}}~\beta_c}^2=-(\mathbf A^{-1}_h	\tilde {\bf p}_M+\tilde u_M{{\bf b}^*_h}, {\bf q}_{RT})_{L^2(\Omega)}=0.\]
Consequently, 
\[
 ~\Pi_0 \tilde {\bf p}_M=-\mathbf A_h\nabla_{NC} \alpha_{CR}-\tilde u_M{\bf b}_h. 
 \]
 The Raviart-Thomas function allows 
for
 div$~\tilde {\bf p}_M={\text {div}}_{NC}~\tilde{\bf p}_M \in P_0(\mathcal T)$ and hence (in 2D),
\[ 
\tilde {\bf p}_M =\Pi_0 \tilde {\bf p}_M +({\text{div}}_{NC}~\tilde {\bf p}_M)(\bullet-{\text{mid}({\mathcal{T}}))/2}. 
 \]
The equation (\ref{eqna2}) is equivalent to
${\text{div}}_{NC}~\tilde {\bf p}_M=f_h-\gamma_h \tilde u_M. $
The combination of the previous identities proves (\ref{pmt}) for $\tilde u_{CR}:=\alpha_{CR}$. A piecewise integration 
by parts of the product of $\tilde {\bf p}_M$
for (\ref{pmt}) with $\nabla_{NC} v_{CR}$ leads to
\[ 
-({\text{div}}_{NC}~\tilde {\bf p}_M, v_{CR})_{L^2(\Omega)}=(\tilde{\bf p}_M, \nabla_{NC} v_{CR})_{L^2(\Omega)}.
 \]
The aforementioned identities for $\Pi_0 \tilde{\bf p}_M$ and ${\text{div}}_{NC}~\tilde {\bf p}_M$ show that this equals
\[ -(f_h-\gamma_h  \tilde u_M, v_{CR})_{L^2(\Omega)}=-(\mathbf A_h\nabla_{NC} \alpha_{CR}+ \tilde u_M{\bf b}_h, \nabla_{NC} v_{CR})_{L^2(\Omega).}\]
This proves (\ref{eqna3}) for $\tilde u_{CR}\equiv \alpha_{CR}$.
To verify (\ref{umt}), the identity (\ref{pmt})  is substituted in (\ref{eqna1}) for some general 
\[       
{\bf q}_{RT}=\Pi_0{\bf q}_{RT}+(\text{div}_{NC}~{\bf q}_{RT})(\bullet-{\text{mid}}({\mathcal{T}}))/2 \in RT_0(\mathcal T).
\]
This shows 
\begin{eqnarray*}
(\text{div}_{NC}~{\bf q}_{RT}, \tilde u_M)_{L^2(\Omega)}&=&(-\nabla_{NC} \tilde u_{CR},{\bf q}_{RT})_{L^2(\Omega)}
+(f_h-\gamma_h \tilde u_M,\frac{S(\mathcal T)}{4} ~{\text{div}}_{NC} ~{\bf q}_{RT})_{L^2(\Omega)}.
\end{eqnarray*}
A piecewise integration by parts shows  $(-\nabla_{NC} \tilde u_{CR},{\bf q}_{RT})_{L^2(\Omega)}=(\tilde u_{CR},\text{div}_{NC} ~{\bf q}_{RT})_{L^2(\Omega)}$
and  hence,
\[
\Big (\tilde u_M\Big(1+\gamma_h \frac{S(\mathcal T)}{4}\Big)-\frac{S(\mathcal T)}{4}f_h-\tilde u_{CR},~\text{div}~{\bf q}_{RT}\Big)_{L^2(\Omega)}=0.
 \]
Since the divergence operator is surjective from $RT_0(\mathcal T)$ onto $P_0(\mathcal T)$ and since  
the previous identity  holds for all ${\bf q}_{RT} \in RT_0(\mathcal T)$), it follows 
\[
\tilde u_M(1+\gamma_h \frac{S(\mathcal T)}{4})=\frac{S(\mathcal T)}{4}f_h+ \Pi_0 \tilde u_{CR}.
 \]
This is equivalent to (\ref{umt}) and concludes the  proof. \qed
%
\subsection{{\it A Priori} Error Estimates for RTFEM }
This subsection establishes well-posedness of the mixed finite element method (\ref{eqna1})-(\ref{eqna2}) and {\it a priori} error estimates for mixed 
formulation (\ref{eq4})  via the equivalence of  RTFEM  and NCFEM. \\
The following theorem deals with the well-posedness of the mixed finite element method (\ref{eqna1})-(\ref{eqna2}) with a more general right hand side.\\
\noindent
For given ${\bf {g_{RT}}} \in RT_0(\mathcal T)$, define ${\bf {g}} \in RT_0(\mathcal T)^{*}$ by
\begin{equation}\label{g-form}
{\bf {g}}({\bf q}):= (\mathbf{A}_h^{-1}{\bf {g}}_{RT}, {\bf q} )_{L^2(\Omega)} + (\text{div}~{\bf g}_{RT}, \text{div}~{\bf q} )_{L^2(\Omega)}~~\text{for all}~ {\bf q}\in RT_0(\mathcal {T}).
\end{equation} 
For  $f_h \in P_0(\mathcal T),$ and ${\bf {g}} \in RT_0(\mathcal T)^{*}$ a modified mixed finite element method reads as: seek $({\bf {p}_{M}}, u_{M}) \in RT_0(\mathcal T) \times  P_0(\mathcal T)$ such that
\begin{eqnarray}
&&(\mathbf{A}_h^{-1}{\bf p_M}+u_M{\bf b}_h^* ,{\bf q} _{RT})_{L^2(\Omega)}-(\text{div}~{\bf q}_{RT},u_M)_{L^2(\Omega)}
                                         = {\bf {g}}({\bf q} _{RT}) \;\;\; \nonumber \\
&&\hspace{6.6cm}\text{for all}~ {\bf q}_{RT}\in RT_0(\mathcal {T}), \label{eqna1-m}\\
&&(\text{div}~ {\bf p_M},v_h)_{L^2(\Omega)}+ (\gamma_h u_M, v_h)_{L^2(\Omega)}=(f_h,v_h)_{L^2(\Omega)} \;\; \text {for all}~  v_h \in P_0(\mathcal{T}). \label{eqna2-m}
\end{eqnarray} 
\begin{thm} \label{thm-stability-m}(Stability)
~For all ${\bf {g}} \in RT_0(\mathcal T)^{*}$ given by (\ref{g-form}) and $f_h \in P_0(\mathcal T),$ the modified mixed finite element problem (\ref{eqna1-m})-(\ref{eqna2-m}) has a unique  solution 
$({\bf {p}_M}, u_{M}) \in RT_0(\mathcal T) \times \in P_0(\mathcal T)$ with
\begin{equation}\label{stability-m}
\norm{({\bf {p}_M}, u_{M})}_{H(\text{\rm div}, \Omega) \times L^2(\Omega)} \lesssim 
\norm{({\bf {g}}, f_h)}_{H(\text{\rm div}, \Omega)^{*} \times L^2(\Omega)}.
\end{equation}
\end{thm}

As in Subsection 4.2, the solution of modified RTFEM (\ref{eqna1-m})-(\ref{eqna2-m}) is represented in terms of the solution of a suitable NCFEM.\\
\noindent 
\textit{Proof.} Since ${\bf {g}}({\bf {q}})$ is given by (\ref{g-form}), the equation (\ref{eqna1-m}) is written 
equivalently
\begin{eqnarray}
 (\mathbf{A}_h^{-1}({\bf p_M}-{\bf {g}}_{RT})&+&  u_M{\bf b}_h^*,{\bf q} _{RT})_{L^2(\Omega)}
 =(\text{div}~{\bf q}_{RT},u_M + \text{\rm div}~{\bf {g}}_{RT} )_{L^2(\Omega)} \; \nonumber\\
 &&\hspace{4cm}\text{for all}\, {\bf q}_{RT}\in RT_0(\mathcal {T}). \label{eqna-m-1}
\end{eqnarray}
Since $-\Pi_0(\mathbf{A}_h^{-1}({\bf p_M}-{\bf {g}}_{RT})+ u_{M}{\bf b}_h^*) \in P_0(\mathcal {T};{\mathbb {R}}^2),$ the discrete Helmholtz decomposition states 
\begin{equation} \label{discrete-HDC}
-\Pi_0(\mathbf{A}_h^{-1}({\bf p_M}-{\bf {g}}_{RT})+  u_M {\bf b}_h^*) = \nabla_{NC} \alpha_{CR} + \rm{Curl}~ \beta_{C}
\end{equation}
for some nonconforming function $\alpha_{CR}\in CR_0^1(\mathcal{T})$ and some $\beta_C\in P_1(\mathcal{T})\cap C(\overline{\Omega}).$  The choice of  ${\bf {q}}_{RT}= \rm{Curl}~\beta_C$ in (\ref{eqna-m-1}) shows  that $ \rm{Curl}~\beta_{C} =0.$ Hence,
 \begin{equation*} 
 \Pi_0 ({\bf p_M}-{\bf {g}}_{RT})  = -\Big(\mathbf{A}_h \nabla_{NC} \alpha_{CR} + u_M {\bf b}_h \Big).
 \end{equation*}
Equation (\ref{eqna2-m}) implies 
\begin{equation}\label{div-pm-grt}
{\text{div}}_{NC}~( {\bf p}_M-{\bf {g}}_{RT}) = f_h-\gamma_h u_M-{\text{div}}_{NC}~{\bf {g}}_{RT}.
\end{equation}
and
\[ 
({\bf p}_M-{\bf {g}}_{RT}) = \Pi_0 ( {\bf p}_M-{\bf {g}}_{RT})+({\text{div}}_{NC}~( {\bf p}_M-{\bf {g}}_{RT}))(\bullet-{\text{mid}({\mathcal{T}})})/2. 
 \]
Hence,
\begin{eqnarray}\label{pm-grt}
({\bf p}_M-{\bf {g}}_{RT}) &=&-\Big(\mathbf{A}_h \nabla_{NC} \alpha_{CR} + u_M {\bf b}_h \Big)\nonumber \\
 &&+(f_h-\gamma_h u_M-{\text{div}}_{NC}~{\bf {g}}_{RT})(\bullet-{\text{mid}({\mathcal{T}}))/2}. 
\end{eqnarray}
For all $v_{CR}\in CR_0^1(\mathcal {T}),$ the last term on the right hand-side of (\ref{pm-grt})
is orthogonal to $\nabla_{NC} v_{CR}$ with respect to $L^2(\Omega)$ inner product. This leads to
\begin{eqnarray*}
(\mathbf{A}_h \nabla_{NC} \alpha_{CR} + u_M {\bf b}_h, \nabla_{NC} v_{CR}) = -({\bf p}_M-{\bf {g}}_{RT},\nabla_{NC} v_{CR}).
\end{eqnarray*}
For the last term on the right-hand side, a piecewise integration with (\ref{div-pm-grt}) yields
\begin{eqnarray}\label{mnc-fem}
(\mathbf{A}_h \nabla_{NC} \alpha_{CR} + u_M {\bf b}_h, \nabla_{NC} v_{CR}) &+& (\gamma_h u_M, v_{CR}) 
=(f_h-{\text{div}}_{NC}~{\bf {g}}_{RT}, v_{CR}).
\end{eqnarray}
A substitution of  (\ref{pm-grt}) in (\ref{eqna-m-1}) with 
${\bf {q}}_{RT}:=\Pi_0 {\bf {q}}_{RT}+({\text{div}}_{NC}~{\bf {q}}_{RT})(\bullet-{\text{mid}({\mathcal{T}})})/2$   and piecewise integration 
$(-\nabla_{NC} \alpha_{CR},{\bf q}_{RT})_{L^2(\Omega)}=(\alpha_{CR},\text{div}_{NC} ~{\bf q}_{RT})_{L^2(\Omega)}$ yields after
some direct calculation
\begin{eqnarray*}
&&(\text{div}_{NC}~{\bf q}_{RT}(1+\frac{S(\mathcal T)}{4})),\text{div}_{NC}~{\bf g}_{RT} )_{L^2(\Omega)} + (u_M, \text{div}_{NC}~{\bf q}_{RT})\\
&&~~~=(\alpha_{CR},\text{div}_{NC}~{\bf q}_{RT})_{L^2(\Omega)}+(\frac{S(\mathcal T)}{4} (f_h-\gamma_h u_M), ~{\text{div}}_{NC} ~{\bf q}_{RT})_{L^2(\Omega)}.
\end{eqnarray*}
Since this holds for all ${\bf q}_{RT} \in RT_0(\mathcal {T}),$ it follows immediately
\begin{equation}\label{um-form}
u_M= (1+\gamma_h \frac{S(\mathcal T)}{4})^{-1} \Big( -(1+\frac{S(\mathcal T)}{4})~\text{div}_{NC}~{\bf g}_{RT}+\frac{S(\mathcal T)}{4}f_h+ \Pi_0 \alpha_{CR}\Big).
\end{equation}
The stability result (\ref{stability}) of Theorem \ref{thm2}  applies to (\ref{mnc-fem}). This implies 
\begin{equation}\label{alpha-estimate}
\tnorm{\alpha_{CR}}_{NC} \lesssim \|{\bf {g}}_{RT}\|_{H(\text{div},~\Omega)} + \|f_h\|.
\end{equation}
From the representations (\ref{um-form}) and (\ref{pm-grt}) of $u_M$ and ${\bf {p}_M},$  
(\ref{alpha-estimate}) proves  stability result (\ref{stability-m}). This concludes the proof.\qed
\noindent
Theorem  \ref{thm-stability-m} implies the   well-posedness 
of the mixed finite element method (\ref{eqna1})-(\ref{eqna2}).
\begin{corollary} (Stability)~
There exists a  unique solution $({\bf {p}_M},u_M)\in RT_0(\mathcal {T})\times P_0(\mathcal {T})$ to the problem (\ref{eqna1})-(\ref{eqna2}) with
\begin{equation}\label{stability-mixed}
\norm{({\bf {p}_M}, u_{M})}_{H(\text{\rm div};\Omega) \times L^2(\Omega)} \lesssim 
\norm{f_h}_{L^2(\Omega)}.
\end{equation}
\end{corollary}

Below, the main theorem of this section is discussed.
\begin{thm} \label{thm3.11}({\it a~priori} error control of RTFEM )
~Under the assumption {\bf (A1)-(A2)} with $u\in H^1_0(\Omega)$ for $f\in L^2(\Omega)$ and for  $\epsilon >0$ 
 with sufficiently small maximal mesh-size $h$,
there exists a unique solution $({\bf p}_M,u_M) \in RT_0(\mathcal {T})\times P_0(\mathcal{T}) $ of the mixed method (\ref{eqna1})-(\ref{eqna2}).
 Further, it holds
 \begin{eqnarray}
\norm{u-u_M} &\lesssim &  (h+ \epsilon^2)\norm{f},\\
  \norm{{\bf {p}}-{\bf p_M}} &\lesssim&  (h+\epsilon) \norm{f}, \label{er_uM}\\
  \norm{\text{\rm div} \:({\bf p}-{\bf p_M})} &\lesssim & \norm{f-f_h} +(h+\epsilon^2) \norm{f}.  \label{er_dpM}
 \end{eqnarray} 
\end{thm}
 The remaining parts of this subsection are devoted to the proof which starts with an error estimate of $\tilde{e}:=u_{CR}-\tilde u_{CR}$.
\begin{lm}\label{lm4.2}(An intermediate estimate)
~Let $u_{CR}$ and $\tilde u_{CR}$ be the solutions of (\ref{non1}) and (\ref{eqna3}), respectively. Then, for sufficiently small maximal mesh-size $h$ 
\begin{eqnarray}\label{etilde2}
  \tnorm{{u_{CR}-\tilde u_{CR}}}_{NC} +\norm{u_{CR}-\tilde u_{CR}} &\lesssim & h \norm{f}.
\end{eqnarray}
\end{lm}
\noindent 
\textit{Proof.} A substitution of (\ref{umt}) in (\ref{eqna3}) and (\ref{non1}) lead for any $v_{CR}\in CR^1_0(\mathcal{T})$ to
\begin{eqnarray}\label{mfem2}
 &&a_{NC}(\tilde e, v_{CR})=(f-f_h,v_{CR})_{L^2(\Omega)}+(\frac{S(\mathcal{T})}{4}\gamma_h  \mathbf{A}_h\nabla_{NC}\tilde u_{CR},\nabla_{NC} v_{CR})_{L^2(\Omega)}\nonumber\\
&&~~+({\bf b}_h\frac{S(\mathcal{T})}{4}f_h,\nabla_{NC} v_{CR})_{L^2(\Omega)}
 - (\gamma_h( \tilde u_{CR}-\Pi_0 \tilde u_{CR}),v_{CR}-\Pi_0 v_{CR} )_{L^2(\Omega)}\nonumber\\
&&~~-((\mathbf{A-A}_h) \nabla_{NC}\tilde u_{CR},\nabla_{NC} v_{CR})_{L^2(\Omega)} -((\mathbf{b-b}_h) \tilde u_{CR},\nabla_{NC} v_{CR})_{L^2(\Omega)}\nonumber\\
&&~~-((\gamma-\gamma_h) \tilde u_{CR},v_{CR})_{L^2(\Omega)}.
\end{eqnarray}
Note that the first term on the right-hand side can be rewritten with $\Pi_0$ and then equals  $(f-f_h, v_{CR}-\Pi_0 v_{CR})_{L^2(\Omega)}.$
The choice of $v_{CR}=\tilde e$ in  (\ref{mfem2}) with an application of  G\r {a}rding's inequality (\ref{non4-coer}), $S(\mathcal{T})\lesssim h^2$ and $\norm{\tilde u_{CR}}\lesssim \tnorm{\tilde u_{CR}}_{NC}$ yields
\begin{align} \label{eq4.15}
 \alpha \tnorm{\tilde e}^2_{NC}-&\beta \norm{\tilde e}^2 \lesssim \Big( osc (f, \mathcal T) +h^2(\norm{\mathbf A_h}_\infty \norm{\gamma_h}_\infty +  \norm{\gamma_h}_\infty) \;\tnorm{\tilde u_{CR}}_{NC} \nonumber\\
 &+ (\norm{\mathbf{A- A}_h}_\infty +\norm{\mathbf{b- b}_h}_\infty)\;\tnorm{\tilde u_{CR}}_{NC} \nonumber\\
 & +h^2\norm{\textbf b_h}_\infty \norm{f_h}  \Big)~\tnorm{\tilde e}_{NC}+
\norm{\gamma- \gamma_h}_\infty\|\tilde u_{CR}\|~\|\tilde e\|.
\end{align}
Since $\norm{\tilde e}\lesssim \tnorm{\tilde e}_{NC},$  an application of (\ref{tilde-u-cr}) shows 
\begin{equation}\label{mfem1}
 \tnorm{\tilde e}_{NC} \lesssim  osc(f, \mathcal T)+\Big(h^2+\norm{\mathbf{A- A}_h}_\infty+\norm{\mathbf{b- b}_h}_\infty+\norm{\gamma- \gamma_h}_\infty\Big)\;\tnorm{\tilde u_{CR}}+\norm{\tilde e}.
\end{equation}
It therefore, remains to estimate $\norm{\tilde e}.$ An appeal to Aubin-Nitsche duality argument 
applied to the dual problem (\ref{adjoint-problem}) plus (\ref{conforming-error-2}) and (\ref{consistency}) lead to
\begin{eqnarray*} \label{mfem3}
 (g,\tilde e)_{L^2(\Omega)}&=&a_{NC}(\tilde e,\Phi-\Phi_C)+(g,\tilde e)_{L^2(\Omega)}-a_{NC}(\tilde e,\Phi)+a_{NC}(\tilde e,\Phi_C)\nonumber \\
&\lesssim&  \tnorm{\tilde e}_{NC} \| \Phi-\Phi_C\|_1 +|a_{NC}(\tilde e, \Phi_C)|\nonumber\\
 &&+\tnorm{\tilde e}_{NC} \,\sup_{0\neq w_{CR} \in CR^1_0(\mathcal{T})} \frac {|a_{NC}( w_{CR},\Phi)- (g, w_{CR})_{L^2(\Omega)}|} {\tnorm{w_{CR}}_{NC}}.
\end{eqnarray*}
For the second last term  on the right-hand side, recall (\ref{mfem2}) with $v_{CR}=\Phi_C$ and proceed as in the proof of the estimate (\ref{eq4.15}) to obtain
\begin{equation}\label{4.17}
 |a_{NC}(\tilde e, \Phi_C)|\lesssim  \Big(osc (f, \mathcal T)+ (h^2 
  +\norm{\mathbf{A- A}_h}_\infty +\norm{\mathbf{b- b}_h}_\infty +\norm{\gamma- \gamma_h}_\infty) \tnorm{\tilde u_{CR}}\Big) \norm{\Phi_C}_1.
\end{equation}
Since $\norm{\Phi_C}_1\lesssim \norm{\Phi}_1 \lesssim \norm{g},$
a substitution of (\ref{conforming-error-2}), (\ref{consistency}) and (\ref{4.17}) in the previous estimates yields
\begin{align}\label{etilde3}
 \norm{\tilde e}= \sup_{0\neq g \in L^2(\Omega)} &\frac{|(g,\tilde e)_{L^2(\Omega)}|}{\norm{g}}\lesssim osc (f, \mathcal T)+ \epsilon \tnorm{\tilde e}_{NC} \nonumber\\
&+\Big(h^2+\norm{\mathbf{A- A}_h}_\infty+\norm{\mathbf{b- b}_h}_\infty
+\norm{\gamma- \gamma_h}_\infty\Big) \tnorm{\tilde u_{CR}}.
\end{align}
Since $\tnorm{\tilde u_{CR}}_{NC} \lesssim \norm{f_h}$ with $\norm{f_h}\lesssim \norm{f}$, (\ref{mfem1}) results in 
\begin{align}\label{mfem11}
 \tnorm{\tilde e}_{NC} &\lesssim  \;osc(f, \mathcal T)\nonumber\\
 &~~+\Big(h^2+\norm{\mathbf{A- A}_h}_\infty+\norm{\mathbf{b- b}_h}_\infty+\norm{\gamma- \gamma_h}_\infty \Big)\norm{f}+\norm{\tilde e}.
\end{align}
For sufficiently small $h$,  $\norm{\mathbf{A- A}_h}_\infty\lesssim h,~
\norm{\mathbf{b- b}_h}_\infty \lesssim h,$ $~\norm{\gamma- \gamma_h}_\infty \lesssim h $   in 
(\ref{etilde3}) leads  to
\begin{align}\label{etilde-3}
 \norm{\tilde e} \lesssim \epsilon \tnorm{\tilde e}_{NC}
& + h\;\norm{f}.
\end{align}
A substitution of (\ref{etilde-3}) in (\ref{mfem1}) results for sufficiently small $h$ in
\begin{equation*}
\tnorm{\tilde e}_{NC} \lesssim   h\;\norm{f}.
\end{equation*}
This and  (\ref{etilde-3}) prove (\ref{etilde2}). \qed

{\textit{Proof of Theorem \ref{thm3.11}.}} 
Uniqueness of a discrete solution  follows from the stability result (\ref{stability-mixed}) with $f_h=0.$
In order to estimate $\norm{u-u_M}$, the definition of $u_M$ in (\ref{umt}) implies
 \begin{align*}
  \norm{u-u_M}&=\norm{(1+\gamma_h S(\mathcal{T})/4)^{-1} \Big((1+\gamma_h S(\mathcal{T})/4)u-(\Pi_0 \tilde u_{CR}+\frac{S(\mathcal{T})}{4}f_h )\Big)} \nonumber\\
 &\lesssim \norm{u-u_{CR}}+\norm{ u_{CR}-\tilde u_{CR}}+\norm{ \tilde u_{CR}-\Pi_0 \tilde u_{CR}}+\norm{\frac{S(\mathcal{T})}{4}(f_h-\gamma_h u)}.
 \end{align*}
 Since $\norm{\tilde u_{CR}-\Pi_0\tilde u_{CR}}\lesssim  h \tnorm{\tilde u_{CR}}_{NC}$ and $S(\mathcal{T})\lesssim h^2,$
 this yields
 \begin{align}\label{mfem7}
  \norm{u-u_M}\lesssim \norm{u-u_{CR}}+\norm{ u_{CR}-\tilde u_{CR}}+ h \tnorm{\tilde u_{CR}}_{NC}+h^2 \norm{f_h-\gamma_h u}.
 \end{align}
A substitution of (\ref{tilde-u-cr})  in (\ref{mfem7}) with  Lemma \ref{lm4.2} and Theorem \ref{theorem-ncm-error} results in
\begin{equation*}\label{umix}
 \norm{u-u_M}\lesssim {osc} (f, \mathcal T)+ ( \epsilon^2 + h)\; \norm{f}.
\end{equation*}
The definition of ${\bf p}$ and (\ref{pmt}) imply
 $$
 {\bf p}-{\bf p_M}=-(\mathbf{A}\nabla {u}+ u{\bf b} )+( \mathbf{A}_h\nabla _{NC} \tilde{u}_{CR}+u_M{\bf b}_h  )-(f_h-\gamma_h u_M )
 (\bullet-\text{mid}(\mathcal T))/{2}.
 $$
 Hence,
 \begin{eqnarray}\label{mfem8}
  \norm{{\bf p}-{\bf p_M}}&\leq& \norm{-(\mathbf {A-A}_h) \nabla u -u(\mathbf {b-b}_h)-\mathbf{A}_h(\nabla u-\nabla_{NC} \tilde{u}_{CR})-(u- u_M ){\bf b}_h}\nonumber\\
&&~+h \norm{f_h-\gamma_h u_M }.
 \end{eqnarray}
 The substitution of $u-\tilde u_{CR}=(u- u_{CR})+( u_{CR}-\tilde u_{CR})$ in (\ref{mfem8}) results in
 \begin{align*}\label{mfem4}
  \norm{{\bf p}-{\bf p_M}}& \lesssim \norm{\mathbf {A-A}_h}_\infty
   \norm{u}_1+\norm{\mathbf {b-b}_h}_\infty \norm{u}+\tnorm{u-u_{CR}}_{NC}\nonumber\\
&~~+\tnorm{u_{CR}-\tilde{u}_{CR}}_{NC}+\norm{u- u_M} +h \norm{f_h-\gamma_h u} +h\norm{u-u_M }.
 \end{align*}
For sufficiently small $h$, Lemma \ref{lm4.2}, Theorem \ref{theorem-ncm-error}, and  (\ref{umix}) imply
 \begin{equation*}
 \norm{{\bf p}-{\bf p_M}}\lesssim {osc} (f, \mathcal T)+ \epsilon \;\norm{f}.
\end{equation*} 
In order to prove the  estimate of $\norm{\text{div}({\bf p}-{\bf p}_M)}$, (\ref{eq3}) and (\ref{eqna2}) together lead to 
\begin{equation*}
  \text{div}({\bf p}-{\bf p_M})=f-f_h-\gamma u+\gamma _h u_M.
\end{equation*}
 Hence,
 \begin{eqnarray}\label{divmix}
  \norm{\text{div}({\bf p}-{\bf p_M})}\leq\norm{f-f_h}+\norm{\gamma -\gamma _h}_\infty \norm{u}+\norm{\gamma_h}_\infty \norm{u- u_M}.
 \end{eqnarray}
A substitution of (\ref{umix}) in (\ref{divmix}) yields (\ref{er_dpM}) and this concludes the proof.\qed
\begin{remk}
With the regularity result $u\in H^{1+\delta}(\Omega)\cap H^1_0(\Omega) $ and  $\epsilon= O (h^{\delta})$, the
error estimates in Theorem  \ref{thm3.11} read
\begin{eqnarray}
\norm{u-u_M} &\lesssim &  h^{\min{(1,2\delta)}}\;\norm{f},\\
  \norm{{\bf {p}}-{\bf p_M}} &\lesssim&  h^{\delta} \norm{f}, \label{er_uM-r}\\
  \norm{\text{\rm div} \:({\bf p}-{\bf p_M})} &\lesssim & \norm{f-f_h} +h^{\min{(1,2\delta)}}\; \norm{f}.  \label{er_dpM-r}
 \end{eqnarray} 
\end{remk}
For related error estimates, when $\delta =1$ see \cite{Doug,Dg} and \cite{Chen}.
\begin{remk}
Note that  for our analysis, only  regularity  estimate for the dual problem in the  broken Sobolev $H^{1+\delta}(\mathcal{T}),$ for
 some $\delta$ with $0 < \delta < 1,$ is required and hence, the assumptions on $\mathbf A$ , ${\mathbf b}$ and $\gamma$ may be weakened in the sense 
that $\mathbf A \in W^{1,\infty}(\mathcal{T};\mathbb R_{sym}^{2\times 2}) $ , $ {\mathbf b}\in W^{1,\infty}(\mathcal{T};\mathbb R^{2})$ and $\gamma \in W^{1,\infty}(\mathcal{T};\mathbb R)$.
Such conditions are more relevant for elliptic interface problems, when the interfaces are aligned to element faces, ( cf. \cite [Sect. 2.4] {N}).
\end{remk}
\section{{\it A Posteriori} Error Control}
This section is devoted to the {\it a posteriori} error  analysis of the mixed finite element scheme 
(\ref{eqna1})-(\ref{eqna2})  to generalize \cite{Car_apo} via the unified approach of \cite{Car1}.

Define ${\bf e_p}:={\bf p}-{\bf {\bf p_M}},$ and $e_u:=u-u_M$. Then, (\ref{eq4}) and (\ref{eqna1})-(\ref{eqna2}) lead to
\begin{eqnarray}
({\mathbf A}^{-1}{\bf e_p}+e_u {\bf b}^*,{\bf q})_{L^2(\Omega)}-(\text{div}~ {\bf q},e_u)_{L^2(\Omega)}&=&\mathcal R_1({\bf q}) ~~ \text{for all}~ {\bf q} \in H(\text{div},\Omega),\label{err1}\\
(\text{div}~ {\bf e_p},v)_{L^2(\Omega)}+ (\gamma e_u,v)_{L^2(\Omega)}&=&\mathcal R_2(v) ~~~ \text{for all} ~	v\in L^2(\Omega).\label{err2}
\end{eqnarray}
Here and throughout this paper  $\mathcal R_1({\bf q})$ and  $\mathcal R_2(v)$ read
 \begin{eqnarray}
  \mathcal R_1({\bf q})&:= &\mathcal R_{11}({\bf q})+\mathcal R_{12}({\bf q}),\label{res1}\\
\mathcal R_2(v)&:= &(f-(\text{div}~ {\bf {\bf p_M}}+\gamma_h u_M)-(\gamma-\gamma_h)u_M,v)_{L^2(\Omega)}\nonumber\\
&=& ((f-\gamma\;u_M)
-\Pi_0(f-\gamma\;u_M),v)_{L^2(\Omega)},\label{res2}\\
\text{~where}~~~~~~ \mathcal R_ {11}({\bf q})&:=&-({\mathbf A}_h^{-1}{\bf {\bf p_M}}+u_M {\bf b}_h^*,{\bf q})_{L^2(\Omega)}+(\text{div}~ {\bf q},u_M)_{L^2(\Omega)},\nonumber\\
\mathcal R_{12}({\bf q})&:=&-(({\mathbf A}^{-1}-{\mathbf A}_h^{-1}){\bf {\bf p_M}}+u_M ({\bf b}^*-{\bf b}_h^*),{\bf q})_{L^2(\Omega)}.\nonumber
 \end{eqnarray}
\subsection{Unified {\it A Posteriori} Analysis}
Theorem \ref{lemma1.1}- \ref{lemma2} imply the  well-posedness of the system (\ref{eq4}) and so the residuals 
$\mathcal R_1,~\mathcal R_2$ of (\ref{res1})-(\ref{res2}) allow for the equivalence  \cite{Car1}
 \begin{equation}
  \norm{{\bf p}-{\bf {\bf p_M}}}_{H(\text{div},\Omega)}+\norm{u-u_M}_{L^2(\Omega)}\approx \norm{\mathcal R_1}_{H(\text{div},\Omega)^*}+\norm{\mathcal R_2}_{L^2(\Omega)}.
 \end{equation}
 The  estimate for $ \displaystyle \mathcal R_2(v)$ 
 reads
\begin{equation} \label{res3}
 \norm{\mathcal R_2}=\norm{f-(\text{div}~ {\bf {\bf p_M}}+\gamma~ u_M)-(\gamma-\gamma_h) u_M} \leq \norm{(1-\Pi_0)(f-\gamma\;u_M)}.
 \end{equation}
Recall that $f_h=\text{div}~  {\bf {\bf p_M}}+\gamma_h ~ u_M $ denotes a piecewise polynomial approximation of $f.$\\
\\
\noindent \textbf{Fortin interpolation operator} \cite[pp 124,128]{Brezzi}. There exists an interpolation operator 
 $$
 I_F:H^1(\Omega;\mathbb R^2)\longrightarrow  RT_0(\mathcal {T}) 
 $$ 
with the orthogonality condition
\begin{equation}\label{int1}
 \int_{\Omega} u_M  ~\text{div}({\bm \phi}-I_F {\bm \phi}) dx=0 \qquad \text{for all}~ {\bm \phi} \in H^1(\Omega; \mathbb R^2)
\end{equation}
and the approximation property
\begin{equation}\label{int2}
 \norm{h_{\mathcal{T}}^{-1}({\bm \phi}-I_F {\bm \phi})} \lesssim \norm{\bm \phi}_{H^1(\Omega)}. 
\end{equation}
\begin{lm}\label{lemma3} (Regular Split)~For any ${\bf q}\in H(\text{\rm div}, \: \Omega)$, there exist ${\bm \phi} \in H^1(\Omega; \mathbb R^2)$ and
$\psi \in H^1(\Omega)$ such that ${\bf q}= {\bm \phi} +\rm{Curl} ~\psi$ in ${\Omega}$ and
\begin{equation}\label{decom}
 \norm{\rm{div}~\phi}+\norm{\nabla\psi} \lesssim \norm{{\bf q}}_{H(\text{\rm div},\Omega)}.
\end{equation}
\end{lm}
\begin{proof}.
 Let ${\bf q}\in H(\text{div},\Omega)$. Extend $\text{div}~{\bf q}|_{\Omega}$  by zero in some ball
 $\mathcal B \supset \supset \Omega.$  Let  $z\in H^2(\mathcal B) \cap H^1_0({\mathcal B})$ be the unique solution of  
 $-\Delta z=\text{div}~ {\bf q} $ in $\Omega$ with $z|_{\partial \mathcal B}=0$.
Also, let  ${\bm \phi}=-\nabla z$, so that 
\[
\norm{\text{div}~\phi}\leq \norm{z}_2 \lesssim \norm{\text{div} ~{\bf q}}\leq \norm{{\bf q}}_{H(\text{div},\Omega)}. 
\]
Since  ${\bm \phi} =-\nabla z$, 
 $\text{div}~ ({\bf q}-{\bm \phi})=0$ in $\Omega,$ and hence, 
 ${\bf q}={\bm \phi} +\text{Curl}~ \psi$ with 
$ \norm{\nabla\psi} = \norm{\text{Curl} ~\psi }=  \norm{{\bf q}- {\bm \phi}}\lesssim \norm{{\bf q}}_{H(\text{div},\Omega)}.$\qed
\end{proof}
\begin{lm}\label{thm3}~There  holds
 \begin{eqnarray*}
  \norm{\mathcal R_1}_{H(\rm {div},\Omega)^*}&\lesssim& \norm{h_{\mathcal{T}}({\mathbf A}_h^{-1}
  {\bf {\bf p_M}}+u_M{\bf b}_h^*)}+\min_{v\in H^1_0(\Omega)}\norm{{\mathbf A}_h^{-1}{\bf p_M}+u_M{\bf b}_h^* -\nabla v}\\
&&+\norm{({\bf A^{-1}-A}_h^{-1}){\bf p_M}}+\norm{u_M({\bf b}^*-{\bf b}_h^*)}.
\end{eqnarray*}
 \end{lm}
\textit{Proof.} For the residual $\mathcal R_{11}({\bf q})$  from (\ref{res1}), the regular decomposition of ${\bf q}\in H(\text{div},\Omega)$
from Lemma \ref{lemma3} and the interpolation operator
$I_F {\bm \phi} \in RT_0(\mathcal {T})\subset \text{Ker}~ \mathcal R_{11}$, lead to
 \begin{align*}
 \mathcal R_{11}({\bf q})&=\mathcal R_{11}({\bm \phi}+\text{Curl}~ \psi)=\mathcal R_{11}({\bm \phi}-I_F {\bm \phi}+\text{Curl}~ \psi)\nonumber\\
 &=-(\mathbf A_h^{-1}{\bf p_M}+u_M{\bf b}_h^*, {\bm \phi}-I_F {\bm \phi})_{L^2(\Omega)}+(u_M, \text{div}~ ({\bm \phi}-I_F {\bm \phi}))_{L^2(\Omega)}\nonumber\\
 &~~~-(\text{Curl}~ {\psi}, {\mathbf A}_h^{-1}{\bf p_M}+u_M{\bf b}_h^*)_{L^2(\Omega)} + (u_M,\text{div~(Curl}~ \psi))_{L^2(\Omega)}.\nonumber
\end{align*}
This and (\ref{int1}) imply 
\begin{eqnarray} \label{term}
 \mathcal R_{11}({\bf q})	&=&-({\mathbf A}_h^{-1}{\bf p_M}+u_M{\bf b}_h^*,{\bm \phi}-I_F {\bm \phi})_{L^2(\Omega)} \nonumber \\
&&- (\text{Curl}~ \psi,{\mathbf A}_h^{-1}{\bf p_M}+u_M{\bf b}_h^*)_{L^2(\Omega)}.
 \end{eqnarray}
 The first term on the right-hand side of \eqref{term} is bounded by 
 \begin{align*}
  |({\mathbf A}_h^{-1}{\bf p_M}+u_M{\bf b}_h^*,{\bm \phi}-I_F {\bm \phi})_{L^2(\Omega)}|&
  \leq \norm{{\mathbf A}_h^{-1}{\bf p_M}+u_M{\bf b}_h^*} \norm{{\bm \phi}-I_F {\bm \phi}}.\nonumber
\end{align*}
The approximation property (\ref{int2}) and Lemma \ref{lemma3} result in
\begin{align}\label{eq11}
|({\mathbf A}_h^{-1}{\bf p_M}+u_M{\bf b}_h^*, {\bm \phi}-I_F {\bm \phi})_{L^2(\Omega)}|
 &\lesssim\norm{h_{\mathcal{T}}({\mathbf A}_h^{-1}{\bf p_M}+u_M{\bf b}_h^*)}\norm{\nabla  {\bm \phi}}\nonumber\\
 &\lesssim\norm{h_{\mathcal{T}}({\mathbf A}_h^{-1}{\bf p_M}+u_M{\bf b}_h^*)}\norm{{\bf q}}_{H({\rm div},\Omega)}.
 \end{align}
Given any $v \in H^1_0(\Omega)$, the second term on the right-hand side of ~\eqref{term} is bounded by
\begin{align}\label{eq12}
 -(\text{Curl}~\psi,{\mathbf A}_h^{-1}{\bf p_M}+u_M{\bf b}_h^*)_{L^2(\Omega)}&= -(\text{Curl}~\psi,{\mathbf A}_h^{-1}{\bf p_M}+u_M{\bf b}_h^*)_{L^2(\Omega)}
  \nonumber\\
&~~~~~+(\text{Curl}~\psi,\nabla v)_{L^2(\Omega)} \nonumber\\
  &\leq \norm{{\mathbf A}_h^{-1}{\bf p_M}+u_M{\bf b}_h^*-\nabla v}\;\norm{\text{Curl}~\psi} \nonumber\\
  &\lesssim\norm{{\mathbf A}_h^{-1}{\bf p_M}+u_M{\bf b}_h^*-\nabla v}\,\norm{{\bf q}}_{H({\rm div}, \:\Omega)}. 
 \end{align}
The combination of (\ref{eq11})-(\ref{eq12}) shows
 \begin{eqnarray}
  \mathcal R_{11}({\bf q})&\lesssim& \Big(\norm{h_{\mathcal{T}}({\mathbf A}_h^{-1}{\bf p_M}+u_M{\bf b}_h^*)}\nonumber \\
&&+  \min_{v\in H^1_0(\Omega)}\norm{{\mathbf A}_h^{-1}{\bf p_M}+u_M{\bf b}_h^*- \nabla v}\Big)\;\norm{{\bf q}}_{H({\rm div}, \:\Omega)}.\label{R-11}
 \end{eqnarray}
 The Cauchy-Schwartz inequality leads to
  \begin{equation}
   \mathcal R_{12}({\bf q})\lesssim 
   \Big(\norm{({\bf A^{-1}-A}_h^{-1}){\bf p_M}}+\norm{u_M({\bf b}^*-{\bf b}_h^*)}\Big) \norm{{\bf q}}_{H({\rm div}, \:\Omega)}. \label{R-12}
  \end{equation}
The estimate (\ref{res1}) follows from (\ref{R-11})-(\ref{R-12}) as
\begin{align*}
  \mathcal R_1({\bf q})\lesssim &\Big ( \norm{h_{\mathcal{T}}({\mathbf A}_h^{-1}
  {\bf {\bf p_M}}+u_M{\bf b}_h^*)}+\min_{v\in H^1_0(\Omega)}\norm{{\mathbf A}_h^{-1}{\bf p_M}+u_M{\bf b}_h^* -\nabla v}\\
&+\norm{({\bf A^{-1}-A}_h^{-1}){\bf p_M}}+\norm{u_M({\bf b}^*-{\bf b}_h^*)}\Big) \norm{{\bf q}}_{H({\rm div}, \:\Omega)}.\hspace{2.7cm} \qed
\end{align*}
Lemma \ref{thm3} and Equation (\ref{res3}) result in the following reliable a posteriori estimate ${\bf \eta}.$ 
\begin{thm}\label{thm3.1}({\it a~posteriori} error control)
~Let $({\bf p},u)$ and $({\bf p_M},u_M)$ solve (\ref{eq4}) and (\ref{eqna1})-(\ref{eqna2}).
Then, it holds
\begin{eqnarray}\label{estimator}
&& \norm{{\bf p}-{\bf {\bf p_M}}}_{H(\rm{div},\Omega)}+\norm{u-u_M}\lesssim
{\bf \eta}:= \norm{(1-\Pi_0)(f-\gamma\;u_M)}\nonumber\\
&&\qquad \qquad ~+\norm{h_{\mathcal{T}}({\mathbf A}_h^{-1}{\bf p_M}+u_M{\bf b}_h^*)}
+\min_{v\in H^1_0(\Omega)}\norm{{\mathbf A}_h^{-1}{\bf p_M}+u_M{\bf b}_h^*-\nabla v}\nonumber \\
 &&\qquad \qquad ~+\norm{({\bf A^{-1}-A}_h^{-1}){\bf p_M}}+\norm{u_M({\bf b}^*-{\bf b}_h^*)}.
\end{eqnarray}
\end{thm}
The following lemma enables a refined {\it a posteriori} error analysis for $\norm{u-u_M}$  and $\norm{{\bf p}-{\bf {\bf p_M}}}.$ 
\begin{lm}\label{lm8}
Let $\tilde u_{CR}$ and $({\bf p_M}, u_M)$ solve (\ref{eqna3}) and (\ref{eqna1})-(\ref{eqna2}), respectively. Then it holds
\begin{equation*}\label{5.16}
 \max \Big{\{}\norm{\nabla_{NC} \tilde u_{CR}}, \norm{\left(f_h-\gamma_h u_M \right)\mathbf{A}_h^{-1} 
 \displaystyle{ \frac{\left({\bf x}-\text{\rm mid}(T)\right)}{2}}}\Big{\}} \leq \norm{\mathbf{A}_h^{-1}{\bf p_M}+u_M{\bf b}_h^*}.
\end{equation*}
\end{lm}
{\it Proof}. 
From (\ref{pmt}), \\ $${\bf A}_h^{-1} {\bf p_M}+u_M{\bf b}_h^*=-\nabla_{NC} \tilde u_{CR}+\left(f_h-\gamma_h u_M \right){\bf A}_h^{-1}
\displaystyle{ \frac{\left({\bf x}-\text{\rm mid}(T)\right)}{2}}.$$
Since 
$\left((f_h-\gamma_h \tilde u_M)\left({\bf x}-\text{mid}(T)\right)/2, \nabla_{NC} \tilde u_{CR} \right) =0,$
the Pythagoras theorem yields
\begin{equation*}
 \norm{{\bf A}^{-1}_h {\bf p_M}+u_M{\bf b}_h^*}^2=\norm{\nabla_{NC} \tilde u_{CR}}^2
 +\norm{\left(f_h-\gamma_h u_M \right)\mathbf{A}_h^{-1} 
\displaystyle{ \frac{\left({\bf x}-\text{\rm mid}(T)\right)}{2}}}^2. \hspace{1cm}\qed
\end{equation*}
A consequence of the Lemma \ref{lm8} and the structure  of $\bf{p_M}$ and $u_M$ is the following bound.
\begin{corollary}\label{corollary-1} It holds
$$ \|h_{\mathcal{T}}\;{\bf {p_M}}\|+ \|h_{\mathcal{T}}\;u_M\| \lesssim \norm{h^2_{\mathcal{T}}\;f_h} + \norm{h_{\mathcal{T}}(\mathbf{A}_h^{-1}{\bf p_M}+u_M{\bf b}_h^*)}.
$$
\end{corollary}
\bigskip
The following theorem concerns on an improved error estimate of $e_u:=u-u_M$ in $L^2$-norm.
\begin{thm} \label{5.5}(Refined error estimates)
~Let $u \in H^1_0(\Omega) $ be the unique weak solution of (\ref{eq2}) and 
 let $({\bf p}_M,u_M) $ be  the solution of (\ref{eqna1})-(\ref{eqna2}).  For sufficiently small maximum mesh size $h,$ 
it holds
\begin{eqnarray}\label{estimator1}
&&\norm{\mathbf{A}^{-{1}/{2}}({\bf p-p_M})}\lesssim {osc}(f, \mathcal T) +{osc}(f-\gamma\; u_M, \mathcal T) \nonumber\\
&&+{\displaystyle{\min_{v\in H^1_0(\Omega)}}\norm{{\mathbf A}_h^{-1}{\bf p_M}+u_M{\bf b}_h^* -\nabla v}}
+\Big( 1+\norm{h^{-1}_{\mathcal{T}}({\bf A-A}_h)}_\infty+ \norm{h^{-1}_{\mathcal{T}}({\bf b-b}_h)}_\infty \nonumber \\
&&+\norm{h^{-1}_{\mathcal{T}}(\gamma- \gamma_h)}_\infty\Big)\norm{h_{\mathcal{T}}(\mathbf{A}_h^{-1}{\bf p_M}+u_M{\bf b}_h^*)} 
+\norm{h^2_{\mathcal{T}}\;f_h} + \norm{h_{\mathcal{T}}\;(f_h-\gamma_h\;u_{M})}\nonumber\\
&&+(\norm{(\mathbf A^{-1}-\mathbf A^{-1}_h){\bf p_M }}+\norm{u_M({\bf b}^*-{\bf b}_h^*)}).
\end{eqnarray}
Provided $u \in H^{1+\delta}(\Omega)$ for some $\, 0 <\delta <1,$ it holds
\begin{eqnarray}\label{estimator2}
 &&\norm{u- u_M}\lesssim  {osc}(f, \mathcal T) +{osc}(f-\gamma\;u_M, \mathcal T)\nonumber\\
&&+{\min_{v\in H^1_0(\Omega)}\norm{h^{\delta}_{\mathcal{T}}({\mathbf A}_h^{-1}{\bf p_M}+u_M{\bf b}_h^* -\nabla v)}}
  +\Big( 1 +\norm{h^{-1}_{\mathcal{T}}({\bf A-A}_h)}_\infty\nonumber\\
&&+ \norm{h^{-1}_{\mathcal{T}}({\bf b-b}_h)}_\infty 
+\norm{h^{-1}_{\mathcal{T}}(\gamma- \gamma_h)}_\infty\Big)\norm{h_{\mathcal{T}}(\mathbf{A}_h^{-1}{\bf p_M}+u_M{\bf b}_h^*)}+\norm{h^{2}_{\mathcal{T}}\;f_h}\nonumber \\
&& + \norm{h^{1+\delta}_{\mathcal{T}}\;(f_h-\gamma_h\;u_{M})}+ (\norm{h^{\delta}_{\mathcal{T}}(\mathbf A^{-1}-\mathbf A^{-1}_h){\bf p_M }}+\norm{h^{\delta}_{\mathcal{T}}\;u_M({\bf b^*}-{\bf b}_h^*)}).
\end{eqnarray}
\end{thm}
{\it Proof.} 
Consider the Helmholtz decomposition ${\bf e_p}={\mathbf A} \nabla z+\text{Curl} ~\beta$ for $z \in H^1_0(\Omega) $
 and $\beta \in H^1(\Omega)/\mathbb R$ with ${\bf e_p= p -p_M}$
\begin{equation}\label{ep11}
 (\mathbf{A}^{-1}{\bf e_p},{\bf e_p})_{L^2(\Omega)}=({\bf e_p},\nabla z)_{L^2(\Omega)}+(\mathbf{A}^{-1}{\bf e_p}, \text{Curl}~\beta)_{L^2(\Omega)}.
\end{equation}
For the  first term  on the right-hand side of (\ref{ep11}), an integration by parts plus  (\ref{err2}) lead to
\begin{align}\label{ep12}
 ({\bf e_p},\nabla z)_{L^2(\Omega)}&=(\text{div}~{\bf e_p},z)={\mathcal{R}}_2(z)-(\gamma (u-u_M),z)_{L^2(\Omega)}\nonumber\\
& =(f-f_h-(\gamma-\gamma_h) u_M,z-\Pi_0 z)_{L^2(\Omega)}-(\gamma e_u,z)_{L^2(\Omega)},\nonumber \\
&\lesssim {osc} (f-\gamma u_{M},{\mathcal{T}})\;\norm{z}_1+\norm{ e_u} \norm{z}.
\end{align}
Given any $v \in H^1_0(\Omega)$, equation (\ref{eq4}) shows
\begin{align}\label{ep13}
&(\mathbf{A}^{-1}{\bf e_p},~ \text{Curl}~\beta)_{L^2(\Omega)}\nonumber\\
& =-(\mathbf{A}_h^{-1}{\bf p_M}+ u_M{\bf b}_h^*,~\text{Curl}~\beta)_{L^2(\Omega)}-(e_u {\bf b^*},~\text{Curl}~\beta)_{L^2(\Omega)}\nonumber \\
&~~~-( ({\bf A}^{-1}-{\bf A}_h^{-1}){\bf p_M}+u_M({\bf b}^*-{\bf b}_h^*),~\text{Curl}~\beta)_{L^2(\Omega)}
+(\nabla v,~\text{Curl}~\beta)_{L^2(\Omega)}\nonumber\\
&\lesssim {\min_{v\in H^1_0(\Omega)}\norm{{\mathbf A}_h^{-1}{\bf p_M}+u_M{\bf b}_h^* -\nabla v}} \norm{\text{Curl}~\beta}
+\norm{e_u} \norm{\text{Curl}~\beta}\nonumber\\
&~~~+(\norm{(\mathbf A^{-1}-\mathbf A^{-1}_h){\bf p_M }}+\norm{u_M({\bf b}^*-{\bf b}_h^*)})\norm{\text{Curl}~\beta}.
\end{align}
The substitution of (\ref{ep12})-(\ref{ep13}) in (\ref{ep11}) plus $\norm{z}\lesssim \norm{z}_1\lesssim \norm{\bf e_p}\lesssim \norm {\mathbf{A}^{-{1}/{2}}{\bf e_p}}$
with  $\norm{\text{Curl}~\beta}\lesssim \norm{\bf e_p}\lesssim \norm {\mathbf{A}^{-{1}/{2}}{\bf e_p}}$ result in
\begin{align}\label{p_esti}
 \norm{\mathbf{A}^{-{1}/{2}}{\bf e_p}}&\lesssim osc(f-\gamma u_{M}, \mathcal {T})
+{\min_{v\in H^1_0(\Omega)}\norm{{\mathbf A}_h^{-1}{\bf p_M}+u_M{\bf b}_h^* -\nabla v}}+ \norm{e_u}\nonumber\\
&~~+\norm{(\mathbf A^{-1}-\mathbf A^{-1}_h){\bf p_M }}+\norm{u_M({\bf b}^*-{\bf b}_h^*)}. 
\end{align}
The estimate of  $\norm{e_u}$ starts with a triangle inequality 
\begin{equation}\label{esti_5}
\norm{e_u}\leq\norm{u-\tilde u_{CR}}+\norm{\tilde u_{CR}-u_M}.
\end{equation}
With $\tilde e= u_{CR}-\tilde u_{CR},$ (\ref{etilde3}) and (\ref{mfem1}) yield (for sufficiently small mesh size $h$) that 
\begin{eqnarray}\label{esti_9}
 \tnorm{\tilde e}_{NC}+\norm{\tilde e} &\lesssim&  osc(f, \mathcal T)+\Big(\norm{h_{\mathcal{T}}}_{\infty}+\norm{h^{-1}_{\mathcal{T}}(\mathbf{A- A}_h)}_\infty+\norm{h^{-1}_{\mathcal{T}}(\mathbf{b- b}_h)}_\infty\nonumber\\
 &+&\norm{h^{-1}_{\mathcal{T}}(\gamma- \gamma_h)}_\infty\Big)\tnorm{h_{\mathcal{T}}\tilde u_{CR}}_{NC}
+ \norm{h^2_{\mathcal{T}}\;f_h}.
\end{eqnarray}
The estimates for $\norm{u-\tilde u_{CR}}$ are derived with the help of (\ref{nc-error0}) and (\ref{esti_9}) and a repeated use of triangle inequality. 
This proves
\begin{eqnarray}\label{esti_1}
 \norm{u-\tilde u_{CR}}&\leq& \norm{u-u_{CR}}+\norm{u_{CR}-\tilde u_{CR}}\nonumber\\
&\lesssim &\epsilon (\tnorm{u-\tilde u_{CR}}_{NC}+ \tnorm{\tilde u_{CR}-u_{CR}}_{NC})
+\norm{u_{CR}-\tilde u_{CR}}\nonumber\\
&\lesssim & \epsilon \norm{\nabla_{NC}(u-\tilde u_{CR})}+ {osc}(f, \mathcal T)+ \norm{h^2_{\mathcal{T}}\;f_h}\nonumber\\
  &&+\Big(\norm{h_{\mathcal{T}}}_{\infty}+\norm{h^{-1}_{\mathcal{T}}(\mathbf{A- A}_h)}_\infty+\norm{h^{-1}_{\mathcal{T}}(\mathbf{b- b}_h)}_\infty
  \nonumber\\&& +\norm{h^{-1}_{\mathcal{T}}(\gamma- \gamma_h)}_\infty\Big)\;\tnorm{h_{\mathcal{T}}\tilde u_{CR}}_{NC}.
\end{eqnarray}
Define ${\bf \tilde p}_{CR}:=-(\mathbf A_h\nabla_{NC}\tilde u_{CR}+u_M{\bf b}_h )$ and    ${\bf p}=-({\bf A}\nabla u+{\bf b} u)$  along with an 
addition and subtraction of the term ${\bf p_M}, ~u_M{\bf b^*} $, $\mathbf A_h^{-1} p_M$. This shows
\begin{align}\label{esti_6}
\norm{{\nabla_{NC}( u-\tilde u_{CR})}}&\leq \norm{\mathbf A^{-1}{\bf e_p }}+\norm{(\mathbf A^{-1}-\mathbf A^{-1}_h){\bf p_M }}
+\norm{\mathbf A^{-1}_h({\bf p_M - \tilde p}_{CR})}\nonumber\\
&~~~+\norm{e_u{\bf b^*}}+\norm{u_M({\bf b}^*-{\bf b}_h^*)}.
 \end{align} 
For the third term on the right-hand side of (\ref{esti_6}), (\ref{pmt}) leads to 
\begin{equation}\label{esti_7}
\norm{{\bf p_M}-{\bf \tilde p}_{CR}}\leq \norm{(f_h-\gamma_h u_M)({\bf x}-\text{mid}(T))}\lesssim \norm{h_{\mathcal{T}}(f_h-\gamma_h\; u_M)}.
\end{equation}
The combination of (\ref{esti_1})-(\ref{esti_7}) results in 
\begin{align*}
\norm{u-\tilde u_{CR}}\lesssim & {osc}(f, \mathcal T)+\epsilon \Big(\norm{\mathbf{A^{-1/2}}{\bf e_p}}+ \norm{e_u}\Big)+\norm{h^2_{\mathcal{T}}\;f_h}+ \epsilon\norm{h_{\mathcal{T}}(f_h-\gamma_h\;u_M)}\nonumber\\
&+\Big(\norm{h_{\mathcal{T}}}_{\infty}+\norm{h^{-1}_{\mathcal{T}}(\mathbf{A- A}_h)}_\infty+\norm{h^{-1}_{\mathcal{T}}(\mathbf{b- b}_h)}_\infty
   +\norm{h^{-1}_{\mathcal{T}}(\gamma- \gamma_h)}_\infty\Big)\;\\
&~~~ \tnorm{h_{\mathcal{T}}\tilde u_{CR}}_{NC}
+ \epsilon (\norm{(\mathbf A^{-1}-\mathbf A^{-1}_h){\bf p_M }}+\norm{u_M({\bf b}^*-{\bf b}_h^*)}).
\end{align*}
To bound $\norm{\tilde u_{CR}-u_M}$ in (\ref{esti_5}), use (\ref{umt}) to obtain
\begin{eqnarray*}
 \norm{\tilde u_{CR}-u_M}&\leq &\left(1+\frac{S(\mathcal{T})}{4} \gamma_h \right)^{-1} \norm{\tilde{u}_{CR}-\Pi_0 \tilde{u}_{CR}+\frac{S(\mathcal{T})}{4}(\gamma_h \tilde{u}_{CR}- f_h )},\nonumber\\
&\lesssim & \norm{h_{\mathcal{T}}\nabla_{NC}\tilde u_{CR}} + \tnorm{h^2_{\mathcal{T}}\tilde{u}_{CR}}_{NC}+\norm{h^2_{\mathcal{T}}\;f_h}.
\end{eqnarray*}
The combination of the previous estimates with  (\ref{esti_5}) and Lemma \ref{lm8} leads to 
\begin{align}\label{esti_8}
\norm{e_u}&\lesssim {osc}(f, \mathcal T)+ \epsilon \Big(\norm{\mathbf{A^{-1/2}}{\bf e_p}}+ \norm{e_u}\Big)+\norm{ h^2_{\mathcal{T}}\;f_h}+ \epsilon\norm{h_{\mathcal{T}}(f_h-\gamma_h\;u_M)}\nonumber\\
&~~~+\Big( 1 +\norm{h^{-1}_{\mathcal{T}}(\mathbf{A- A}_h)}_\infty+\norm{h^{-1}_{\mathcal{T}}(\mathbf{b- b}_h)}_\infty
   +\norm{h^{-1}_{\mathcal{T}}(\gamma- \gamma_h)}_\infty\Big)\;\nonumber \\
&~~~~\norm{h_{\mathcal{T}}(\mathbf{A}_h^{-1}{\bf p_M}+u_M{\bf b}_h^*)}
+ \epsilon \Big(\norm{(\mathbf A^{-1}-\mathbf A^{-1}_h){\bf p_M }}+\norm{u_M({\bf b}^*-{\bf b}_h^*)}\Big).
\end{align}
For sufficiently small mesh size $h$,  (\ref{esti_8}) and (\ref{p_esti}) prove (\ref{estimator1}). 
The proof of  (\ref{estimator2}) utilizes the  additional regularity with $\epsilon =O(h^{\delta})$.
\qed
\begin{remk} Corollary \ref{corollary-1} and (\ref{estimator1})-(\ref{estimator2}) yield
\begin{eqnarray*}
\norm{(\mathbf A^{-1}-\mathbf A^{-1}_h){\bf p_M }}+\norm{u_M({\bf b}^*-{\bf b}_h^*)} 
&\lesssim & 
\Big(\norm{h^{-1}_{\mathcal{T}}(\mathbf A^{-1}-\mathbf A^{-1}_h)}_{\infty}+\norm{h^{-1}_{\mathcal{T}}({\bf b}^*-{\bf b}_h^*)}_{\infty}\Big)\\
& & \Big(
\norm{h^2_{\mathcal{T}}\;f_h } 
+ \norm{h_{\mathcal{T}}(\mathbf{A}_h^{-1}{\bf p_M}+u_M{\bf b}_h^*)}
\Big).
\end{eqnarray*}
\end{remk}
Then, estimates can be used in (\ref{estimator1})-(\ref{estimator2}) to provide better estimates in Theorem \ref{5.5}. 
\subsection{Efficiency}
This section  is devoted to prove that the error estimator $\eta$ yields lower bounds for the error in the mixed finite element approximation.
\begin{thm}\label{theorem4.4}(Efficiency) 
 ~Under the assumptions ({\bf {A1}})-({\bf {A2}})   it holds
\begin{align}
  \min_{v\in H^1_0(\Omega)}&\norm{{\mathbf A}_h^{-1}{\bf p_M}+u_M{\bf b}_h^*-\nabla v}+\norm{h_{\mathcal T}({\mathbf A}_h^{-1}{\bf p_M}+u_M{\bf b}_h^*)} \nonumber\\
 &\lesssim\norm{u-u_M}+\norm{{\bf p}-{\bf p_M}}+\norm{(\mathbf A^{-1}-\mathbf A^{-1}_h){\bf p_M}}+\norm{u_M({\bf b}^* -{\bf b}_h^* )  }.\nonumber
 \end{align}
\end{thm}
{\it Proof.}
Step 1 of the proof utilizes  $v:=-u$, and the definition ${\bf p}=-\mathbf A \nabla u+{\bf b} u$  to verify
\begin{align*}
& \min_{v\in H^1_0(\Omega)}\norm{\textbf A^{-1}_h{ \bf p_M} +u_M{\textbf b_h^*} -\nabla v}\leq \norm{\textbf A^{-1}_h{ \bf p_M} +u_M{\textbf b_h^*} +\nabla u}\\
&~~=\norm{\textbf A_h^{-1}({ \bf p-p_M})} +\norm{{(u-u_M)\textbf b^*}}+\norm{(\mathbf A^{-1}-\mathbf A^{-1}_h){\bf p_M}}+\norm{u_M({\textbf b^*} -{\textbf b_h^*} )  }\\
&~~\lesssim \norm{ {\bf p-p_M}} +\norm{u-u_M}+\norm{(\mathbf A^{-1}-\mathbf A^{-1}_h){\bf p_M}}+\norm{u_M({\textbf b^*} -{\textbf b_h^*} )  }.
\end{align*}
In step 2, define the function ${\textbf q_T}:= b_T(\mathbf A^{-1}_h{ \bf p_M} +u_M{\textbf b_h^*} )\in P_4(T)\cap W^{1,\infty}_0(T)$ 
and the cubic bubble function
$
 ~b_T = 27\lambda_1 \lambda_2 \lambda_3 \in P_3(T)\cap C_0(T)
$
in terms of the barycentric coordinates $\lambda_1, \lambda_2, \lambda_3$ of $T \in \mathcal T$\cite{verf_book}.
Since ${\mathbf A}_h^{-1}{\bf p_M}+u_M{\bf b}_h^*$ is affine on $T\in \mathcal T$, an equivalence of norm argument shows
\begin{align*}
 \norm{\mathbf A^{-1}_h{ \bf p_M} +u_M{\textbf b_h^*} }_{L^2(T)}^2&\lesssim  \int_T{\textbf q_T}\cdot({ \mathbf A^{-1}_h\bf p_M} +u_M{\textbf b_h^*} )dx.
 \end{align*}
The definition of ${\bf p}$ and (\ref{eq3}) show that 
\begin{align*}
\norm{{\mathbf A^{-1}_h \bf p_M} +u_M{\textbf b_h^*} }_{L^2(T)}^2 
& \lesssim \int_T{\textbf q_T}\cdot \big ( \mathbf A^{-1}({ \bf p_M - p}) -(u-u_M){\textbf b^*} \big)dx\nonumber\\
& ~~+ \int_T{\textbf q_T}\cdot\Big ( (\mathbf A^{-1}_h-\mathbf A^{-1}){\bf p_M}-u_M({ \bf b^* -b}_h^*)  \Big)~dx \\
&~~-\int_T{\textbf q_T}\cdot \nabla u ~dx.                      
 \end{align*}
The Cauchy inequality and $\norm{{\textbf q_T}}_{L^2(T)}\lesssim \norm{\mathbf A^{-1}_h{\bf p_M}+u_M{\textbf b_h^*} }_{L^2(T)}$ is employed in the first 
two terms. An integration by parts with $\nabla u_M|_T=0$ shows 
in the last term that 
\begin{eqnarray*}
 h_T^2\norm{{\mathbf A^{-1}_h }{\bf p_M}+u_M{\bf b}_h }^2_{L^2(T)} &\lesssim &h_T\norm{\mathbf A^{-1}_h{\bf p_M}+u_M{\bf b}_h^* }_{L^2(T)}\Big (h_T \norm{{\bf p-p_M}}_{L^2(T)}
\nonumber\\
 &&+h_T\norm{u-u_M}_{L^2(T)}+ h_T\norm{(\mathbf A^{-1}-\mathbf A^{-1}_h){\bf p_M}}\nonumber\\
&&+h_T\norm{u_M({\bf b^*-b}_h^*) }_{L^2(T)}\Big) + h_T^2 \int_T (u-u_M)\text{div}~{\bf q}_T dx.
 \end{eqnarray*}
Since $\textbf q_T \in P_4(T)$, an inverse estimate yields
\begin{equation*}
 h_T\norm{\text{div}~ \textbf q_T}_{L^2(T)}\lesssim \norm{\textbf q_T}_{L^2(T)}\lesssim \norm{\mathbf A^{-1}_h{\bf p_M}+u_M{\bf b}^*_h }_{L^2(T)}.
\end{equation*}
Since $ h_T \lesssim 1$, it follows
\begin{align*}
h_T\norm{\mathbf A^{-1}_h{\bf p_M}+u_M{\bf b}_h^* }_{L^2(T)} &\lesssim \norm{u-u_M}_{L^2(T)}+\norm{{\bf p}-{\bf p_M}}_{L^2(T)}\\
&~~+\norm{(\mathbf A^{-1}-\mathbf A^{-1}_h){\bf p_M}}_{L^2(T)}+\norm{u_M({\textbf b^*} -{\textbf b_h^*} )  }_{L^2(T)}.
\end{align*}
 The sum over all triangles implies
 \begin{align*}
h_{\mathcal T}\norm{\mathbf A^{-1}_h{\bf p_M}+u_M{\bf b}_h^* }&\lesssim \norm{u-u_M}+\norm{{\bf p}-{\bf p_M}}\nonumber\\
&~~+ \norm{(\mathbf A^{-1}-\mathbf A^{-1}_h){\bf p_M}}+\norm{u_M({\textbf b^*} -{\textbf b_h^*} )  }.
 \end{align*} 
This concludes the  rest of the proof. \qed
\section{Computational Experiments} 
This section is devoted to validation of theoretical results by numerical experiments and to test the performance of the adaptive algorithm. 
\subsection{Practical Implementation}
The adaptive finite element algorithm  starts with the initial coarse triangulation $\mathcal T_0$, followed by the procedures
 {\bf SOLVE, ESTIMATE, MARK and REFINE}  for different  levels $\ell= 0, 1, 2, \cdots$.\\
\\
{{\bf SOLVE.}}~The discrete solution 
$({\bf p_\ell},u_\ell)\in RT_0(\mathcal T_{\ell})\times P_0(\mathcal T_{\ell})$
 of (\ref{eqna1}-\ref{eqna2}) is computed on  each level $\ell$ with the corresponding triangulation $\mathcal T_{\ell}$ 
and  basis functions as prescribed in~\cite{Car2}.\\
{\bf  ESTIMATE.} The estimator  $\eta_{\ell} $  is defined in (\ref{estimator}).  
In the estimator term $\norm{{\mathbf A}_h^{-1}{\bf p_\ell}+u_\ell {\bf b}_h^*-\nabla v}$,
 the function $v$ is chosen by  post processing $\tilde u_{CR}$, that is  
$v=-{\mathcal A}\tilde u_{CR}$, where the averaging operator 
${\mathcal A}: CR^1(\mathcal T)\rightarrow P_1(\mathcal T)$ \cite{Carhop} is defined by\\
 $$v(z):={\mathcal A} \tilde u_{CR}(z):= \sum_{T\in \mathcal T(z)} \frac{\tilde u_{CR}|_T(z)}{|\mathcal T(z)|}\qquad \text{for all}~ z \in \mathcal N . $$
$|\mathcal T(z)|$ denote the cardinality of the triangles sharing node $z$.\\
{\bf MARK.} For $0<\theta \leq 1 ,$ compute a minimal subset $\mathcal M_{\ell} \subset \mathcal T_{\ell}$ for red refinement such that 
$$\theta \eta_{\ell}^2 \leq \eta_{\ell}^2(\mathcal M_{\ell} )=\sum_{T\in \mathcal M_{\ell} } \eta_{T,\ell}^2.\qquad $$
{\bf  REFINE.} The new triangulation $\mathcal T_{\ell+1}$ is generated using red-blue-green refinement of the marked elements.
\begin{remk}
 In the process of computation of the  solution, 
the given function $f$ over each element is approximated by the integral mean $f_h=\frac{1}{|T|}\int_T f(x) dx$. 
The integrals $\int_T f(x) dx$ are computed by one-point numerical quadrature rule over the element, that is,   $|T|f(\text{mid(T)})$, where $|T|$ denotes an area of element $T$ and
mid($T$) is the centroid of the element. For the edge integral with  Dirichlet condition $u_D$ simple one point integration 
reads  $\int_E u_D ds \approx|E|u_D(\text{mid}(E))$, where $|E|$ denotes the length
 of edge and $(\text{mid}(E))$, the midpoint of the edge.
 \end{remk}
 \begin{remk}
 Let $({\bf p},u)$ and $({\bf p_M},u_M)$ solve (\ref{eq4}) and (\ref{eqna1})-(\ref{eqna2})
and let $e_u :=\norm{u-u_M}$ and  ${\bf e_p}:=\norm{{\bf p-p_M}}$.
With the number of unknowns  {\rm Ndof} $(\ell)$ and the error $e(\ell)$ on the level $\ell$,
the experimental order of convergence is defined by  
\begin{equation*}
CR(e)=\frac{\log(e(\ell-1)/e(\ell))}{\log({\rm Ndof}(\ell)/{\rm Ndof}(\ell-1))} ~~{\text{for}} ~e_u,  {\bf e_p}, \text{and} ~\eta.
\end{equation*}
\end{remk}
\begin{example}\label{example1}
Consider the PDE (\ref{eq1}) with  coefficients  $~ \mathbf A=I,~~\mathbf {b} =(r \cos\theta,r \sin\theta)$ \text{and} $\gamma=-4$  with Dirichlet boundary condition on the L-shaped domain
$\Omega =(-1,1)\times (-1,1)\setminus [0,1]\times [-1,0]$ and the exact solution (given in polar coordinates)
$$u(r,\theta)=r^{2/3} \sin\big( 2\theta/3 \big).$$
\end{example}
%
For the given parameters, conditions of \cite[Theorem 3.1]{CHZ} are {\em not} satisfied. Utilizing  their notation, 
  $b_1({\bf q},v):=-(v,\text{div} ~{\bf q})_{L^2(\Omega)}+(\tilde {\bf b} v,{\bf q})_{L^2(\Omega)}$ with $\tilde {\bf b}=\mathbf A^{-1}{\bf b}$, for $v=|\Omega|^{-\frac{1}{2}}$ 
\begin{equation}
 \beta_1\leq \sup_{{\bf q}\in H(\text{div},\Omega)/\{0\}}\frac{\norm{\text{div}~ {\bf q}}+\norm{\tilde{\bf b}v}~\norm{{\bf q}}}{\norm{{\bf q}}_{H(\text{div},\Omega)}}
\leq \sqrt{1+\int_\Omega|{\bf x}|^2 dx}\leq \sqrt 3
\end{equation}
since $|{\bf x}|\leq \sqrt{2}$ for all $x\in \Omega$. It is relatively straightforward to verify $\alpha\leq\norm{a}=1$ (in the notation of \cite{CHZ}) and hence
$\alpha\norm{a}^{-2}\beta_1^2-\gamma\leq3-4<0$ (notice that the coefficient $\gamma=-4$ in \cite[pp 224-225]{CHZ} is different from the parameter $\gamma=4$ in 
\cite[Equation (3.3)]{CHZ} and this might give reasons for confusion). This violates the (implicit) condition $\delta_1\geq 0$ in \cite[Equation (3.1)]{CHZ}.
\begin{table}[!h]
\begin{center}
 \begin{tabular}{ | c| c | c |c |c |c |c |}      \hline
 $N$  & $~~e_u$  &$CR(e_u)$  & $~~{\bf e_p}$    &$CR({\bf e_p}$) &$~~\eta$       &$CR(\eta)$    \\ \hline \hline
  68 &0.16656920 &                 &0.26578962   &                &1.01064602  & \\ \hline 
256  &0.08258681 &0.5292           &0.19505767   &0.2333          &0.52572088   &0.4930\\ \hline 
992  &0.04098066 &0.5173           &0.12772995   &0.3125          &0.27713363   &0.4726\\ \hline 
3904 &0.02034316 &0.5111           &0.08188794   &0.3244          &0.14883131   &0.4537\\ \hline 
15488&0.01011251 &0.5072           &0.05215656   &0.3273          &0.08185377   &0.4338\\ \hline 
61696&0.00503450 &0.5046           &0.03310369   &0.3289          &0.04621899   &0.4135\\ \hline 
\end{tabular}
\end{center}
\caption{Errors and the experimental convergence rates for uniform  mesh refinement}
\end{table}
\begin{table}[!h]
\begin{center}
 \begin{tabular}{ |c | c | c |c|c |c |c |}
      \hline
 $N$  & $~~e_u$  &$CR(e_u)$  & $~~{\bf e_p}$    &$CR({\bf e_p}$) &$~~\eta$       &$CR(\eta)$    \\ \hline \hline
  68  &0.16656920   &       &0.265789390  &        &1.01064602&\\ \hline
  196 &0.09911109   &0.4904 &0.196603070  & 0.2848 &0.63780403&0.4348\\ \hline
  453 &0.06588355   &0.4874 &0.128212606  & 0.5102 &0.41616295&0.5096\\ \hline
  987 &0.04198085   &0.5786 &0.089068850  & 0.4677 &0.27834036&0.5164\\ \hline
 2348 &0.02897814   &0.4277 &0.057982998  & 0.4953 &0.18977893&0.4419\\ \hline
 5039 &0.01921399   &0.5380 &0.040735672  & 0.4617 &0.12698725&0.5261\\ \hline
11342 &0.01265778   &0.5144 &0.026826168  & 0.5154 &0.08633161&0.4756\\ \hline
24118 &0.00874275   &0.4905 &0.018141078  & 0.5185 &0.05808281&0.5253\\ \hline
50952 &0.00583392   &0.5408 &0.012484994  & 0.4999 &0.04006535&0.4965\\ \hline
 \end{tabular}
\end{center}
\caption{Errors and the experimental convergence rates for adaptive mesh refinement}  
\end{table}  
\begin{figure}[H]
\centering
 \subfloat[]{\includegraphics[width=6.0cm]{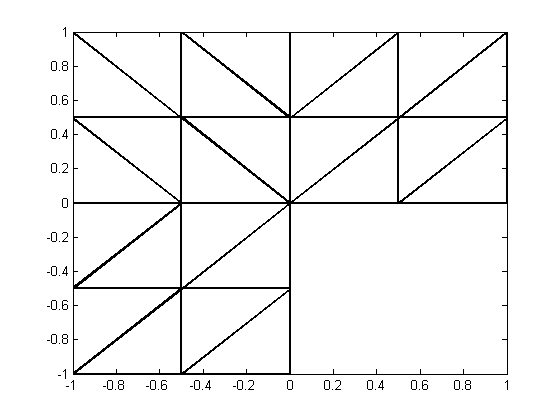}}
 \subfloat[ ]{\includegraphics[width=6.0cm]{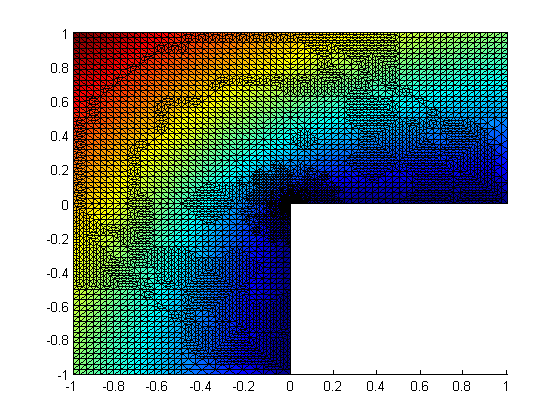}}\\
 \subfloat[ ] {\includegraphics[width=3.0in]{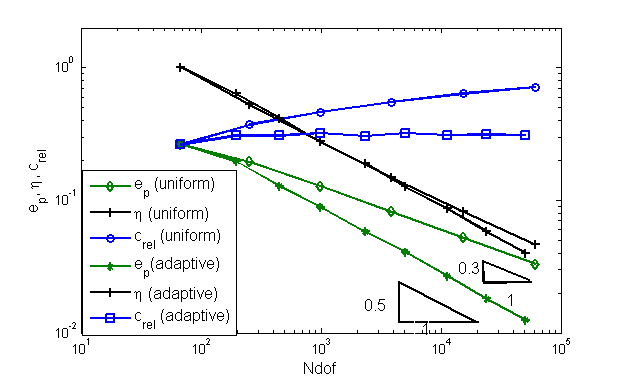}}
\caption{(a) Initial triangulation $\mathcal T_0$ (b) Discrete solution $u_M$  for adaptive  mesh-refinement 
(c)  Ndof vs. ${\bf e_p} $,  $\eta$ and $C_{rel}$   }
\end{figure}
Tables 1 and 2 show the errors and experimental convergence rate for uniform and adaptive mesh-refinements.
 ~Figure 1(a) denotes the initial triangulation $\mathcal T_0$ with $h\approx 0.5$.
Figure  1(b) depicts   the discrete solution $u_M$ and illustrates the  adaptive mesh-refinement near the singularity. 
In Figure 1(c),  a convergence history for the error ${\bf e_p}$ and the estimator $\eta$ is plotted as a function of the number of degrees of freedom
for the cases of  uniform and adaptive mesh-refinement of the non-convex L-shaped domain.
%
%
Adaptive mesh refinement gives an optimal empirical convergence  rate of order $ 0.5$ for ${\bf e_p}$, while standard uniform refinement achieves
suboptimal empirical convergence rate $\approx 0.33$ as expected from the theory.
For  both the cases, $C_{rel}$,
 the  ratio between the error and the estimator  is also plotted.
\begin{example}\label{example3}
{\it Crack problem}: Consider the PDE (\ref{eq1}) with coefficients $\mathbf A=I, ~{\bf b}=(x-1,y+1)$ and $\gamma=0$ on
$~\Omega =\{(x,y)\in \mathbb{R}^2:|{{\bf x}}|\leq 1\setminus[0, ~1]\times \{0\}\}$ with Dirichlet boundary condition and
  exact solution $~u(r,\theta)=r^{{1}/{2}} \sin{\theta}/{2} -{r^2}/{2} \sin ^2(\theta)$ ( in polar coordinates).
\end{example}
The problem is  called  non-coercive   \cite{nochetto}, since $ (\gamma-\frac{1}{2}\nabla \cdot{\bf b})< 0$. Figure 2(a) shows the discrete solution $u_M$
 along with the adaptive mesh-refinement. 
Note that the mesh is strongly refined near the singularity at the origin.
 The results are summarized in Figure 2(b) and displays  convergence rates for the error  ${\bf e_p}$ and the {\it a~posteriori} estimator $\eta$.
It is observed that a suboptimal empirical convergence rate of $ 0.25$  for uniform mesh-refinement and an improved  optimal empirical
 convergence rate of 0.5 for adaptive mesh-refinement are achieved. In this case,
  $C_{rel}$ is close to 0.5.\\

\begin{figure}[!h]
\centering
 \subfloat[]{\includegraphics[width=5.4cm]{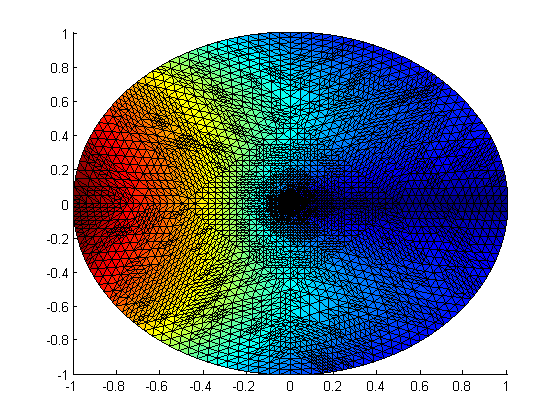}}
 \subfloat[]{\includegraphics[width=6.2cm]{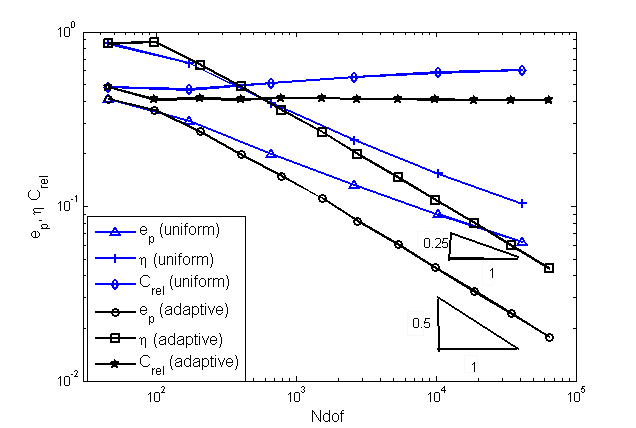}}
\caption{(a) Discrete solution $u_M$ for adaptive refinement   
(b) Ndof vs  ${\bf e_p} $,  $\eta$  and $C_{rel}$ }
\end{figure}
 \begin{example}\label{example4}
 Consider  the PDE (\ref{eq1}) with coefficients $\mathbf A= I, ~{\bf b}=(0,0)$ for different values of $\gamma$ 
and Dirichlet boundary conditions on the L-shaped domain.
\end{example}
Since the first Laplace eigenvalue for the  L-shaped domain  $\lambda_1 \approx$ 9.6397238440219, 
 the coefficients lead to the Laplace 
operator  with positive and negative eigenvalues.
\noindent 
\begin{figure}[!h]
\centering
 \subfloat[]{\includegraphics[width=6.0cm]{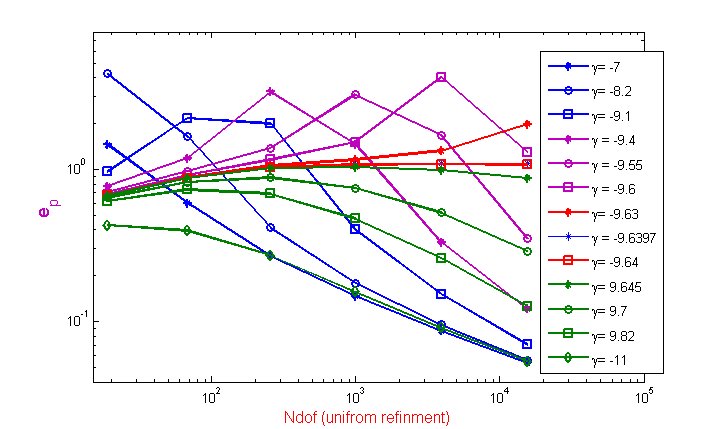}}
 \subfloat[]{\includegraphics[width=6.0cm]{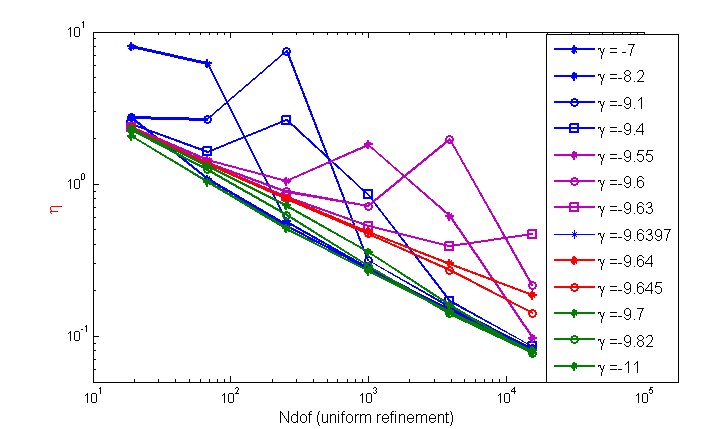}}
\caption{(a) $e_p$ and  (b) $\eta 	$ for different $\gamma$ with uniform refinement}
\end{figure}
\begin{figure}[!h]
\centering
 {\includegraphics[width=7.0cm]{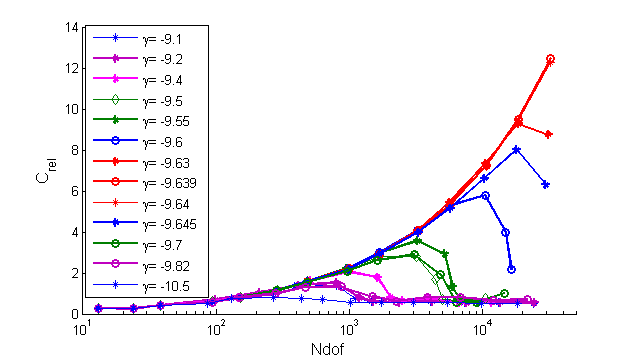}}
\caption{  $C_{rel}$ for different $\gamma$ with adaptive refinement}
\end{figure}
The  fact that the convergence  is sensitive to the smallness of the discretization parameter $ h $ is clearly observed in Figure 3(a).
This observation holds true for conforming, nonconforming and mixed finite element methods.
Figure 3(b) depicts that the estimator mirrors the error behavior.
This is also true for the case of adaptive refinement.\\
 Figure 4 plots the reliability constant
$C_{rel}$ for  various values of $\gamma$ close to the eigenvalue $\lambda_1$ vs the number of degrees of freedom. Note that $C_{rel}$ is sensitive to
  the discretization parameter $h$ especially when $\gamma$ is closer to $\lambda_1$. Thus,  a sufficiently small mesh-size is a crucial requirement for 
the well-posedness and the convergence of the solution.
\subsection{Conclusions }
 From  the numerical experiments, it is observed that efficiency index  lies between 2 and 3.5 for both
  uniform and adaptive triangulations. This confirms the efficiency of {\it a~posteriori} error control for
  non-smooth problems defined in non-convex domains.\\
The overall assumption on the mesh-size to be sufficiently small is in fact crucial
in practice, as shown in the third example  empirically.
\begin{acknowledgements} The first author acknowledges the support of National Program on Differential Equations:
Theory, Computation \& Applications (NPDE-TCA) vide Department of Science \& Technology (DST) Project No. SR/S4/MS:639/09 during his visit to IIT Bombay.
The second author acknowledges the financial support of  Council of Scientific and Industrial Research (CSIR), Government of India.
\end{acknowledgements}



\end{document}